\documentclass{amsart}

\usepackage{longtable}
\usepackage{tikz-cd}
\usepackage{amsmath,amssymb,amscd,amsthm,amsfonts}
\usepackage[all,cmtip]{xy} 
\usepackage{hyperref}
\usepackage [numbers] {natbib}
\usepackage{enumerate}

\newtheorem{thm}{Theorem}[section]
\newtheorem{lemma}[thm]{Lemma}
\newtheorem{prop}[thm]{Proposition}
\newtheorem{cor}[thm]{Corollary} 
\theoremstyle{definition}
\newtheorem{defn}[thm]{Definition} 
\theoremstyle{remark}

\newtheorem{rmk}[thm]{Remark}
\numberwithin{equation}{section}

\newcommand{\uhp}{\mathcal{H}}

\newcommand{\fP}{\mathfrak{P}}

\newcommand{\fa}{\mathfrak{a}}

\newcommand{\fc}{\mathfrak{c}}

\newcommand{\fn}{\mathfrak{n}}

\newcommand{\fp}{\mathfrak{p}}

\renewcommand{\AA}{\mathbb{A}}
\newcommand{\CC}{\mathbb{C}}

\newcommand{\PP}{\mathbb{P}}
\newcommand{\QQ}{\mathbb{Q}}
\newcommand{\RR}{\mathbb{R}} 
\newcommand{\ZZ}{\mathbb{Z}}

\newcommand{\A}{\mathbb{A}}
\newcommand{\F}{\mathbb{F}}
\newcommand{\Q}{\mathbb{Q}}
\newcommand{\R}{\mathbb{R}} 
\newcommand{\Z}{\mathbb{Z}}

\newcommand{\cA}{\mathcal{A}}

\newcommand{\cH}{\mathcal{H}}
\newcommand{\cI}{\mathcal{I}}
\newcommand{\cL}{\mathcal{L}}
\newcommand{\cM}{\mathcal{M}}
\newcommand{\cN}{\mathcal{N}}
\newcommand{\cO}{\mathcal{O}}

\renewcommand{\c}[1]{\mathcal{#1}}

\newcommand{\til}{\widetilde}
\newcommand{\trpz}[1]{{{}^{\mathrm{t}}\negthinspace{#1}}}

\DeclareMathOperator{\Cl}{\mathrm{Cl}}

\DeclareMathOperator{\coker}{\mathrm{coker}}

\DeclareMathOperator{\GL}{\mathrm{GL}}
\DeclareMathOperator{\SL}{\mathrm{SL}}

\DeclareMathOperator{\PGL}{\mathrm{PGL}}
\DeclareMathOperator{\PSL}{\mathrm{PSL}}
\DeclareMathOperator{\tr}{\mathrm{tr}}

\DeclareMathOperator{\Hom}{\mathrm{Hom}}
\DeclareMathOperator{\End}{\mathrm{End}}
\DeclareMathOperator{\Gal}{\mathrm{Gal}}

\DeclareMathOperator{\Lie}{\mathrm{Lie}}
\DeclareMathOperator{\Mat}{\mathrm{M}}

\newcommand{\isom}{\cong}

\DeclareMathOperator{\norm}{\mathbf{N}}

\DeclareMathOperator{\NS}{\mathrm{NS}}

\DeclareMathOperator{\vol}{\mathrm{vol}}

\DeclareMathOperator{\lmod}{\backslash}

\newcommand{\lsup}[2]{{{}^{#1}\negmedspace{#2}}}
\newcommand{\smatrix}[4]{\ensuremath\bigl( \begin{smallmatrix}
#1&#2\\ #3&#4
\end{smallmatrix} \bigr)}
\newcommand{\smat}[4]{\ensuremath\bigl( \begin{smallmatrix}
#1&#2\\ #3&#4
\end{smallmatrix} \bigr)}
\newcommand{\bmat}[4]{\ensuremath
  \begin{pmatrix}
    #1&#2\\ #3&#4
  \end{pmatrix}}

\newcommand{\splt}{\mathrm {split}}
\newcommand{\ram}{\mathrm {ram}}
\numberwithin{equation}{section}
\title{Virtual Abelian Varieties of {$\GL_2$}-type}
\author{Chenyan Wu}
\address{Shanghai Center for Mathematical Sciences\\ Fudan University\\ 220 Handan Rd, Shanghai, China, 200433 and School of Mathematics and Statistics\\ University of Melbourne\\ VIC 3010, Australia}
\email{chenyan.wu@unimelb.edu.au}
\thanks{This research is supported in part by the National Natural Science Foundation of China (\#11601087), by the Program of Shanghai Academic/Technology Research Leader (\#16XD1400400) and by the General Program of National Natural Science Foundation of China (\#11771086).}
\date{\today}
 
\begin{document}

\begin{abstract}
  This paper studies a class of Abelian varieties that are of $\GL_2$-type and with isogenous classes defined over a number field $k$. We treat the cases when their endomorphism algebras are either (1) a totally real field $K$ or (2) a totally indefinite quaternion algebra over a totally real field $K$. Among the isogenous class of such an Abelian variety, we identify one whose Galois conjugates can be described in terms of actions of Atkin-Lehner operators and the class group of $K$. Thus we deduce that such Abelian varieties are parametrised by finite quotients of certain PEL Shimura varieties. These new families of moduli spaces are further analysed when they are of dimension $2$. We provide explicit numerical bounds for when they are surfaces of general type. In addition, for two particular examples, we show that they are both rational surfaces by computing the coordinates of inequivalent elliptic points and studying the intersections of Hirzebruch cycles with exceptional divisors.
\end{abstract}
\maketitle{}

\section*{Introduction}
\label{sec:introduction}
In \cite{MR1212980}, Ribet considered the modularity problem of  elliptic curves defined over
$\bar {\QQ}$ whose $\Gal (\bar {\QQ}/\QQ)$-conjugates are all isogenous to each other.  The case of
CM elliptic curves with such a property was first studied by Gross\cite{MR563921} who coined the
name, $\QQ$-curve. Ribet showed that given a non-CM $\QQ$-curve $C$ there exists a simple Abelian variety
$A$ of $\GL_2$-type (Definition~\ref{defn:GL-2-type})  defined over
$\QQ$ having $C$ as a $\bar {\QQ}$-quotient. Let $E$ denote the endomorphism algebra of $A$. It must be a totally real number field in this case. Recall that the Tate-$\ell$-module $V_\ell A$ of $A$ is
free of rank $2$ over $E\otimes_\QQ \QQ_\ell$. Let $\lambda$ be a prime of $E$ lying above $\ell$
and set $V_\lambda A = V_\ell A \otimes_{E\otimes_\QQ \QQ_\ell} E_\lambda$. Then the Galois action
of $\Gal (\bar {\QQ}/\QQ)$ on $A$ gives rise to a $\lambda$-adic representation of degree $2$. The
question whether a $\QQ$-curve is modular reduces to showing the modularity of the $\lambda$-adic representations
associated to $A$. An affirmative answer is furnished by Serre's conjecture on mod-$\ell$-representations of $\Gal (\bar{\QQ}/\QQ)$, which
was  proved by Khare-Wintenberger\cite{MR2480604}. Thus Abelian varieties of $\GL_2$-type, as well as their geometric quotients, are of particular interest in the study of modularity. 

Let $k$ be a number field. Instead of $\QQ$-curves, one may as well consider Abelian varieties $B$
over $\bar {k}$ whose $\Gal (\bar {k}/k)$-conjugates are isogenous to $B$. They are the object of
study in this paper in which we generalise several pieces of related work, in the hope that a
version of Serre's conjecture on modularity of mod-$\ell$-representations of $\Gal (\bar{k}/k)$
becomes available in the future. Our main objective is to construct  moduli spaces for these Abelian varieties. We study Galois orbits of the Abelian varieties and relate them to the orbits under the actions of class group of the centre of endomorphism algebra and  Atkin-Lehner operators.  Much of the difficulty we encounter arises from having more complex structure of endomorphism algebras and from having fields with non-trivial class groups. We also estimate the Chern numbers of the moduli spaces and compute two examples. This can potentially lead to explicit examples of Abelian varieties which may provide a test ground for the many conjectures involving Abelian varieties, for example, the BSD conjecture.   We will now review the literature and further discuss our results.

Whereas all elliptic curves are automatically of $\GL_2 (\QQ)$-type, there are many more possibilities for the
endomorphism algebras of Abelian varieties. We will focus on the sub-maximal case, namely (non-CM)
Abelian varieties of $\GL_2$-type. Even after this restriction there are still two possibilities
which are commonly known as the case of real multiplication (RM) and the case of quaternionic
multiplication (QM)
(c.f. Proposition~\ref{prop:endom-alg-of-simple-av-GL2-with-analytic-rep-restriction-or-defined-over-R}).  For
ease of exposition, we first define $k$-virtuality which captures the notion of having isogenous
Galois conjugates.
\begin{defn}
  An Abelian variety $B$ over $\bar {k}$ is said to be $k$-virtual if for all
  $\sigma\in \Gal (\bar {k} /k)$, there exists an isogeny
  $\mu_\sigma: \lsup{\sigma}{B} \rightarrow B$ such that for all $\alpha\in \End(B)$, we have
  $\alpha \circ \mu_\sigma = \mu_\sigma\circ \lsup{\sigma}{\alpha} $.
\end{defn}
For non-CM elliptic curves, the requirement on compatibility with endomorphism ring is vacuous
and in fact, $\QQ$-curve is a short hand for $\QQ$-virtual elliptic curve. The departure from the
more traditional nomenclature is to clarify that the Abelian varieties are only `virtually' defined
over $k$ rather than truly defined over $k$.

In \cite{MR2058652}, Pyle extended the result on relation between $\QQ$-virtual elliptic curves and
Abelian varieties over $\QQ$ of $\GL_2$-type to that on  relation between $\QQ$-virtual Abelian varieties of
$\GL_2$-type and Abelian varieties over $\QQ$ of $\GL_2$-type and in \cite{MR2916969}, Guitart  generalised the result to that over arbitrary number field, but he only considered geometric quotients of the  Ribet-Pyle varieties which have number fields as endomorphism algebra. We first establish an analogous result when the Abelian variety of $\GL_2$-type has QM. We actually prove the result in a  uniform way for both RM and QM case (c.f. Cor.~\ref{cor:simple-factor-of-GL2AV-is-k-virtual-GL2AV}). This shows that the study of virtual Abelian varieties of $\GL_2$-type can be transferred to Abelian varieties of $\GL_2$-type and vice versa.

We also extend Elkies's work\cite{MR2058644} on the construction of moduli spaces of $k$-virtual
elliptic curves to the case of $k$-virtual Abelian varieties of $\GL_2$-type. For each prime $\ell$
of $\QQ$ (the endomorphism algebra of a non-CM elliptic curve $C$), Elkies associated an $\ell$-local tree
to $C$ where, roughly speaking, the vertices represent isomorphism classes of elliptic curves and
the edges represent primitive $\ell$-isogenies. Applying a graph theoretic argument, he showed that
in the isogenous class of $C$ there exists one elliptic curve $C_0$ whose Galois conjugates are
controlled by a certain level structure on $C_0$. The observation is that the Galois orbit of $C_0$,
which is, a priori, difficult to describe, is actually contained in  its Atkin-Lehner orbit. Thus
the moduli spaces are Atkin-Lehner quotients of  modular curves with certain level structure and the $k$-rational points
give rise to $k$-virtual elliptic curves.  For simple $k$-virtual Abelian varieties $B$ of
$\GL_2$-type and for a prime $\lambda$ of the centre $K$ of the endomorphism algebra of $B$, we can
construct a $\lambda$-local tree in an analogous way except that the class group of $K$ now plays a
subtle role. Again our formulation treats both the RM and QM cases largely uniformly. We show in
Theorems~\ref{thm:moduli_of_virtual_av_E}, \ref{thm:moduli_of_virtual_av_D} that the moduli spaces
are quotients of PEL Shimura varieties by the group which is an extension of the Atkin-Lehner group by the class group of $K$ and the $k$-rational points
give rise to $k$-virtual Abelian varieties of $\GL_2$-type. We note that Guitart and Molina\cite{MR2588775} worked out the moduli spaces of virtual QM Abelian surfaces and showed that they are Atkin-Lehner quotients of Shimura curves. As the centre of the endomorphism algebra in their case is $\QQ$ which has class number $1$, our case is much more complicated.

The moduli spaces of $\QQ$-virtual elliptic curves have been well-studied. Elkies\cite{MR2058644}
produced some explicit equations for his moduli spaces which are quotients of modular
curves. Gonz\'alez-Lario\cite{MR1643280} classified those that are of genus $0$ or $1$. Based on
their parametrisation, Quer\cite{MR1770611} computed explicit equations of some $\QQ$-curves.  We
attempt to classify our moduli spaces. At this point, we focus on Abelian surfaces. For moduli
spaces of $k$-virtual RM Abelian surfaces of $\GL_2 (E)$-type, the PEL Shimura varieties are
disjoint unions of Hilbert modular surfaces.   We then go on to analyse the moduli
spaces along the line started by Hirzebruch\cite{MR0393045} and extended by Hirzebruch-Van de
Ven\cite{MR0364262} and Hirzebruch-Zagier\cite{MR0480356}. A thorough write-up is available in the
book of Van der Geer\cite{MR930101}.  Our family of Hilbert modular surfaces has not been considered in the literature. In this paper, as a first step, we treat only the case when $E$ has
trivial narrow class group and when the level structure is $\cO_E/\fp$ for some prime $\fp$ of
$\cO_E$, and leave the more technical/interesting cases for the future. Based on the previous results, we are able to estimate the Chern
numbers of the desingularisation of our Hilbert modular surfaces and to determine explicit bounds on
the discriminant of $E$ and the size of the level structure beyond which the Hilbert modular surfaces
are of general type (Theorem~\ref{thm:general_type}). By Lang's conjecture, we do not expect them to
furnish many $k$-rational points. Thus we turn to examine the Hilbert modular surfaces associated to
$E=\QQ (\sqrt {5})$ with $\fp= (2)$ and $E=\QQ (\sqrt {13})$ with $\fp= (4 + \sqrt {13})$. (See Section~\ref{sec:examples} for precise description of these two Hilbert modular surfaces.) By
studying configuration of rational curves coming from desingularisation and Hirzebruch cycles on these two Hilbert modular surfaces, we conclude that
they are both rational surfaces. In the process we have computed the explicit coordinates of the
inequivalent elliptic points which for $E=\QQ (\sqrt {13})$ should be new. The method is due to
Gundlach\cite{MR0229579}. However as the discriminant increases, the  domain in which one scans for elliptic points grows much larger than a fundamental domain. Thus determining inequivalent ones becomes much harder. Further analysis of the moduli spaces will be part of our future research
topic.

The structure of the article is as follows. In Sec.~\ref{sec:av-GL2-type}, we  describe the possible endomorphism algebras for
Abelian varieties of $\GL_2$-type and show that the geometric  factors of simple Abelian varieties over
$k$ of $\GL_2$-type are
$k$-virtual Abelian varieties of $\GL_2$-type.
In Sec.~\ref{sec:moduli} we determine moduli spaces of $k$-virtual
Abelian varieties of $\GL_2$-type by extending Elkies's local tree constructions. We show that the
moduli spaces are quotients of Hilbert modular surfaces or quaternionic Shimura varieties by
the extension of Atkin-Lehner group by a class group. In Sec.~\ref{sec:class-hilb-modul}, we analyse the cusp and quotient
singularities of the Hilbert modular surfaces in question and estimate their Chern numbers to show
that most of them are of general type. Finally we give two examples where the moduli spaces are
rational surfaces in Sec.~\ref{sec:examples}.

\section{Virtual Abelian Varieties of $\GL_2$-type}
\label{sec:av-GL2-type}
In this section we define the virtual Abelian varieties of $\GL_2$-type and deduce some preliminary results.  We introduce the notions of $\GL_2$-type and virtuality separately.
\subsection{Endomorphism Algebras of Abelian Varieties of $\GL_2$-type}
\label{sec:endom-abel-vari}

Let $k$ be a field of characteristic $0$ and $\overline{k}$ its algebraic closure. In this article,
$k$ is most often a number field. Let $A$ be an Abelian variety over $k$. Write $\End (A) $ for its
endomorphism ring.  The endomorphisms are required to be defined over $k$.  The ring
$\End (A_{\overline{k}})$ consists of all potential endomorphisms of $A$. The endomorphism algebra
$\End^0 (A)$ is defined to be $\End (A) \otimes_\Z \Q$. Let $E$ be a number field. We will consider
those Abelian varieties $A$ that admit a $\Q$-algebra embedding $E\hookrightarrow \End^0(A)$.

\begin{defn}\label{defn:GL-2-type}
  An Abelian variety $A$ defined over $k$ is said to be of $\GL_2 $-type if for some number field
  $E$ such that $[E:\Q] = \dim A$, there is an embedding of $\Q$-algebras
  $E\hookrightarrow \End^0 (A)$. If the number field $E$ is specified, we say that $A$ is of
  $\GL_2(E)$-type.
\end{defn}
 We do not require $\End^0 (A)$ to be isomorphic to $E$, as we intend to study moduli spaces of Abelian varieties where Abelian varieties with bigger endomorphism algebras arise naturally and they form special cycles.
Now we  make a more general definition.
\begin{defn}
  An Abelian variety $A$ defined over $k$ is said to be of $\GL_n (D) $-type if for some division algebra $D$ over $\QQ$
  such that $[D:\Q] = 2\dim A/n$, there is an embedding of $\Q$-algebras
  $D\hookrightarrow \End^0 (A)$.
\end{defn}
We note that in this case the Tate module $V_\ell (A)$ is free of rank $n$ over $D\otimes_\QQ \QQ_\ell$. This is the rationale behind the naming.

We would like to focus on the non-CM Abelian varieties. We make precise what we mean by CM. `Potentially CM' is probably  more correct, but we opt for a shorter name here.
\begin{defn}
  An Abelian variety $A$ defined over $k$ is said to be of CM-type if for some CM algebra $E$ such
  that $[E:\Q] = 2\dim A$, there is an embedding of $\Q$-algebras
  $E\hookrightarrow \End^0 ( A_{\bar {k}})$.
\end{defn}

\begin{rmk}
  \begin{enumerate}
  \item Every elliptic curve is automatically of $\GL_2(\Q)$-type.
  \item Sometimes we simply say that $E$ acts on $A$ when we mean that $E$ acts on $A$ up to
    isogeny.
  \end{enumerate}
\end{rmk}

The requirement of having a big number field acting on an Abelian variety is very strong. We
investigate its implication.  Assume that the Abelian variety $A$ is isogenous to
\begin{equation*} \prod_{i=1}^{m} A_i^{r_i}.
\end{equation*}
where $A_i$'s are simple Abelian varieties which are pairwise non-isogenous. Fix a polarisation of
$A$. Then we have the associated Rosati involution on the endomorphism algebra of $A$. The
endomorphism algebra $\End^0(A_i)$ is a division algebra classified by Albert. We refer to the book
of Mumford \cite{MR2514037} for details. Set $D_i=\End^0(A_i)$ and let $K_i$ denote the centre of
$D_i$ and $K_{i,0}$ the set of fixed points of the Rosati involution. Put $e_i=[K_i:\Q]$ and
$d_i^2=[D_i: K_i]$. The degree $[K_i:K_{i,0}]$ is either $1$ or $2$. Marking the relative degrees on
the diagram, we have
\begin{equation*}
  \begin{tikzcd} \Q \ar[rr, hook, "e_i"] \ar[dr,hook] & &K_i \ar[r,hook,"d_i^2"] & D_i \\
    &K_{i,0} \ar[ur,hook,"\text{$1$ or $2$}"']& & 
  \end{tikzcd}.
\end{equation*}
Composing the embedding $E\hookrightarrow \End^0(A)\isom \prod_{i=1}^m M_{r_i}(D_i)$ with projection
onto each factor $M_{r_i}(D_i)$, we get embeddings $E\hookrightarrow M_{r_i}(D_i)$ for all $i$. A
maximal subfield of $M_{r_i}(D_i)$ has degree $r_ie_id_i$ over $\Q$. In addition, the following
constraints are in effect: $e_id_i | \dim A_i$ if $D_i$ is of type I, II or III;
$e_id_i^2 | 2\dim A_i$ if $D_i$ is of type IV. The types are as in \cite[page 187]{MR2514037}. Briefly, an endomorphism algebra of type I is a totally real number field, that of type II is a
totally indefinite quaternion algebra over a totally real number field, that of type III is a
totally definite quaternion algebra over a totally real number field and that of type IV is a
division algebra over a CM field. Write $\deg E$ for $[E:\Q]$. Note that
\begin{align*}
  \deg E = &\dim A \ge r_i\dim A_i;\\ r_ie_id_i \ge &\deg E
\end{align*}
for all $i$. Thus if for any $i$, $e_id_i \le \dim A_i$, we are forced to have $m=1$ and $A$ is
isogenous to $A_1^{r_1}$ with $e_1d_1 = \dim A_1$. In this case $A_1$ has action by a field of
degree equal to $\dim A_1$. Thus $A_1$ is a simple Abelian variety of $\GL_2$-type. Now
suppose for all $i$, $e_id_i > \dim A_i$. This can happen only when all $D_i$'s are of type IV with
$d_i=1$ and $e_i=2\dim A_i$. In other words all $A_i$'s have CM. We have
\begin{align*}
  2r_i\dim A_i \ge \deg E = \dim A = \sum_{j=1}^m r_j\dim A_j
\end{align*}
for all $i$. Thus $r_i\dim A_i \ge r_j\dim A_j$ for all $i$ and $j$. As a result
$r_i\dim A_i = r_j \dim A_j$ for all $i$ and $j$ and $\deg E = m r_1\dim A_1$. Hence $m\le 2$. When
$m=1$, $A$ is isogenous to $A_1^{r_1}$ which is a power of a CM Abelian variety and $E$ is not a
maximal field acting on $A$. When $m=2$, $A$ is isogenous to $A_1^{r_1} \times A_2^{r_2}$ with $A_1$
and $A_2$ being CM Abelian varieties such that $r_i\dim A_i = r_j \dim A_j$, $E$ is a maximal field acting
on $A$ and furthermore $E$ is a finite field extension of a CM field.

If furthermore we assume that $E$ is stabilised by the Rosati involution on $A$. Then by positivity
of Rosati involution, $E$ is either a totally real field with Rosati involution acting as identity
or a CM field with Rosati involution acting as complex conjugation.

We have shown:
\begin{prop} \label{prop:GL2-av-decomp} Let $A$ be an Abelian variety of $\GL_2$-type over $k$.
  \begin{enumerate}
  \item If $A$ is not a CM Abelian variety, then $A$ is isogenous to $A_1^r$ where $A_1$ is a simple
    Abelian variety of $\GL_2$-type and $r \in \Z_{>0}$.
  \item If $A$ is a CM Abelian variety, then $A$ is isogenous either to $A_1^r$ where $A_1$ is a
    simple CM Abelian variety and $r \in \Z_{>0}$ or to $A_1^{r_1} \times A_2^{r_2}$ where $A_i$ is
    a simple CM Abelian variety and $r_i \in \Z_{>0}$ for $i=1, 2$ and $r_1\dim A_1 = r_2\dim A_2$.
  \end{enumerate}
  
\end{prop}
\begin{rmk} Since obviously $E$ also acts on $A_{\bar{k}}$, we also get a description of the
  decomposition of $A$ over $\bar{k}$.
\end{rmk}

Now we focus on simple Abelian varieties of $\GL_2(E)$-type. Their endomorphism algebra can be
strictly larger than $E$. Let $D$ denote $\End^0 (A)$, $K$ the centre of $D$ and $K_0$ the set of
fixed points in $K$ of the Rosati involution.  Put $e=[K:\Q]$, $e_0= [K_0:\Q]$ and
$d^2=[D: K]$.
\begin{prop}\label{prop:endom-alg-of-simple-av-GL2}
  Let $A$ be a simple Abelian variety of $\GL_2(E)$-type over $k$. Let $g=\dim A$. Then the endomorphism algebra of $A$
  must be of one of the following forms.
  \begin{enumerate}
  \item $D=K=E$ is a totally real number field.
  \item $D$ is a division quaternion algebra over a totally real field $K$ with $[K:\Q]=g/2$ and $E$
    is a quadratic extension of $K$ contained in $D$.
  \item $D=K=E$ is a CM field.
  \item $D$ is a division quaternion algebra over a CM field $K$ with $[K:\Q]=g/2$ and $E$ is a
    quadratic field extension of $K$ contained in $D$.
  \item $D=K$ is a CM field with $[K:\Q]=2g$ and $E$ is a subfield of $K$ with $[K:E]=2$.
  \end{enumerate}
  Furthermore, if $k$ is algebraically closed, then $D$ cannot be of type III (totally definite
  quaternion division algebra over a totally real number field).
\end{prop}
\begin{proof} If $D$ is of type I, II or III, then we have the constraint $ed|g$. When $D$ is of
  type I, then $d=1$ and $e|g$. Thus we must have $e=g$ and $K=E$. When $D$ is of type II or III,
  then $d=2$ and $2e|g$. A maximal subfield of $D$ is of degree $2e$. We must have $2e=g$ and $E$
  must be a quadratic extension of $K$. Of course, this can only happen when $g$ is even.

  Now assume
  that $D$ is of type IV. We have the constraint $e_0d^2 | g$. A maximal subfield of $D$ is of
  degree $2e_0d$. Thus $2e_0d \ge g$. We must have $d=1$ or $2$. When $d=1$, we deduce from $e_0 |g$
  and $2e_0 \ge g$ that $e_0 = g/2$ or $g$. In the former case, we get $D=K=E$ and this can only
  occur when $g$ is even. In the latter case we get that $D=K$ is a CM field with $[K:\Q]$=2g and
  $E$ is a subfield of $K$ with $[K:E]=2$. When $d=2$, we deduce from $4e_0 |g$ and $4e_0 \ge g$
  that $e_0 = g/4$. This can only occur when $4|g$. In this case, $D$ is a division quaternion
  algebra over $K$ which is CM with $[K:\Q]=g/2$ and $E$ is a quadratic extension of $K$ contained
  in $D$.

  When $k$ is of characteristic $0$ and is algebraically closed, then we can rule out more
  possibilities. By \cite[Proposition~15]{MR0156001}, $D$ cannot be of type III; $\End^0 (A)$ is
  forced to grow larger. In fact, $A$ is isogenous to $A_1^2$ with $A_1$ CM.
\end{proof}

If furthermore we assume $k=\CC$, certain analytic representations of $D$ on the Lie algebra of $A$
cannot occur. We summarise the results of \cite [Sec.~4] {MR0156001}. When $D$ is of type I, II or
III, the rational representations of $D$ must contain all of its inequivalent irreducible
representations with the same multiplicity. Thus $E$ acts on $\Lie (A)$ via all of its
embeddings into $\CC$ with each occurring once. Now assume that $D$ of type IV. Then we have
\begin{equation*}
  D \otimes_\Q \R \isom M_d (\CC) \times \cdots \times M_d (\CC)
\end{equation*}
where the product is $e_0$-fold or indexed by the $e_0$ embeddings of $K_0$ into $\R$. The $e_0$
natural projections account for half the number of the inequivalent irreducible representations of
$D$. Denote these by $\chi_\nu$ for $\nu =1, \ldots , e_0$. Then $\chi_\nu$ and $\bar {\chi}_\nu$
account for all the inequivalent irreducible representations of $D$. Let $r_\nu$ (resp. $s_\nu$) be
the multiplicity of $\chi_\nu$ (resp. $\bar {\chi}_\nu$) occurring in the analytic representation of
$D$.  We note that in our case, $r_\nu + s_\nu = 2g/de$ which is $2$ or $1$. Then \cite [Prop.~14,
18, 19] {MR0156001} says that if $\sum r_\nu s_\nu = 0$ or $r_\nu = s_\nu =1$ for all $\nu$, then
cases (3) and (4) in Prop.~\ref{prop:endom-alg-of-simple-av-GL2} cannot occur.

Assume that $A$ and all of its endomorphisms can be defined over $\R$. Then the analytic
representation of $D$ on $\Lie (A_\R)$ must be such that $r_\nu = s_\nu =1$. Then cases (3) and (4)
in Prop.~\ref{prop:endom-alg-of-simple-av-GL2} do not occur. A CM Abelian variety cannot be defined
over a totally real number field, so case (5) is not possible for such $A$.

Assume that each of the embedding of $E$ into $\CC$ occurs exactly once in the analytic
representation of $E$ on $\Lie (A_\CC)$. Then this also forces that $r_\nu = s_\nu =1$, ruling out cases (3) and
(4). In case (5) which is the case of CM Abelian variety, for each conjugate pair of embedding of
$K$ into $\CC$, exactly one of them occurs. In order for each embedding of $E$ to occur, $E$ has to
be the totally real subfield $K_0$ of $K$.

Summarising the above, we get:
\begin{prop}\label{prop:endom-alg-of-simple-av-GL2-with-analytic-rep-restriction-or-defined-over-R}
  Let $A$ be a simple complex Abelian variety of $\GL_2 (E)$-type over a number field $k$. Let $g=\dim A$. Assume one of the following.
  \begin{enumerate}[(a)]
  \item Each of the embedding of $E$ into $\CC$ occurs exactly once in the analytic representation
    of $E$ on $\Lie (A_\CC)$.
  \item $A$ and all of its endomorphisms can be defined over $\R$.
  \end{enumerate}
  Then we have exactly the following possibilities.
  \begin{enumerate}
  \item $D=K=E$ is a totally real number field.
  \item $D$ is a totally indefinite division quaternion algebra over a totally real number field $K$
    with $[K:\Q]=g/2$ and $E$ is a quadratic extension of $K$ contained in $D$.
  \item $D=K$ is a CM field with $[K:\Q]$=2g and $E$ is the totally real subfield $K_0$ of $K$. This
    case does not occur when we assume (b).
  \end{enumerate}
\end{prop}

\subsection{Virtual Abelian Varieties}
\label{sec:abs-simple-factor}

We give the definition of virtuality first.
\begin{defn}
  Let $F$ be a Galois extension of $k$ contained in $\bar {k}$. An Abelian variety $B$ over $F$ is
  said to be $k$-virtual if every element of $\End (B_{\bar {k}})$ can be defined over $F$ and for
  all $\sigma\in \Gal (F/k)$, there exists an isogeny $\mu_\sigma: \lsup{\sigma}{B} \rightarrow B$
  such that for all $\alpha\in \End^0 (B)$,
  $\alpha \circ \mu_\sigma = \mu_\sigma\circ \lsup{\sigma}{\alpha} $.
\end{defn}

Such Abelian varieties arise, for example, in the following fashion.
\begin{lemma}
  Let $A$ be a simple Abelian variety over $k$ such that $A_{\bar {k}}$ is isogenous to $B^r$ where $B$ is a simple Abelian variety over $\bar {k}$. Then $B$ is a $k$-virtual Abelian variety. 
\end{lemma}

\begin{proof}
 Fix an isogeny
$f: A_{\bar {k}} \rightarrow B^r$. Let $\sigma\in \Gal (\bar {k}/k)$. Then we have
\begin{equation*}
  \lsup {\sigma} {B}^r \xleftarrow {\lsup {\sigma} {f}} \lsup {\sigma} {A}_{\bar {k}} \xrightarrow {i_\sigma} A_{\bar {k}}
  \xrightarrow {f} B^r,
\end{equation*}
where $i_\sigma$ is the canonical isomorphism.  Thus by uniqueness of decomposition,
$\lsup {\sigma} B$ is isogenous to $B$.  Let $D=\End^0 (A)$ and $K$ be the centre of $D$. Let
$D'=\End^0 (B)$ and $K' = Z (\End^0 (B))$ be its centre. We have the embeddings
\begin{equation*}
  \begin{tikzcd}
    D \ar[r, hook] & M_{r} (D') \\
    K \ar[u,hook] & K' \ar[l,hook'] \ar[u,hook]
  \end{tikzcd}.
\end{equation*}
As every endomorphism $\alpha$ in $D$ is defined over $k$, we have
$\lsup {\sigma} {\alpha} = i_\sigma^{-1}\circ\alpha \circ i_\sigma$. Now let $\alpha\in K'$. This
can be viewed as an endomorphism (up to isogeny) of $B^r$ by acting diagonally. We note that
$ f^{-1}\circ \alpha\circ f$ lies in $K$, so it is defined over $k$. Thus
\begin{align*}
  &\alpha\circ f \circ i_\sigma \circ \lsup {\sigma} {f}^{-1}
    = f\circ (f^{-1}\circ \alpha\circ f) \circ i_\sigma \circ \lsup {\sigma} {f}^{-1}\\
  = &f\circ (i_\sigma \circ \lsup{\sigma}{f}^{-1} \circ \lsup{\sigma}{\alpha} \circ \lsup{\sigma}{f}\circ i_\sigma^{-1}) \circ i_\sigma \circ \lsup {\sigma} {f}^{-1}
      = f \circ i_\sigma \circ \lsup{\sigma}{f}^{-1} \circ \lsup{\sigma}{\alpha}.
\end{align*}
This means that the following diagram commutes:
\begin{equation*}
  \begin{tikzcd} [column sep=large]
    \lsup{\sigma}{B}^r \ar [r, "f\circ i_\sigma \circ \lsup {\sigma} {f}^{-1}"] \ar [d,"\lsup {\sigma} {\alpha}"] & B^r \ar [d,"\alpha"]\\
    \lsup {\sigma} {B}^r \ar [r, "f\circ i_\sigma \circ \lsup {\sigma} {f}^{-1}"] & B^r
  \end{tikzcd}.
\end{equation*}
This induces a commutative diagram
\begin{equation*}
  \begin{tikzcd}
    \lsup{\sigma}{B} \ar [r, "\mu_\sigma"] \ar [d,"\lsup {\sigma} {\alpha}"] & B \ar [d,"\alpha"]\\
    \lsup {\sigma} {B} \ar [r, "\mu_\sigma"] & B
  \end{tikzcd}
\end{equation*}
where $\mu_\sigma$ is an isogeny induced by $f\circ i_\sigma \circ \lsup {\sigma} {f}^{-1}$. In this
sense, $\mu_\sigma$ is $K'$-equivariant.

Now we  augment $K'$-equivariance to $D'$-equivariance.
  We have a morphism of central simple algebras
  \begin{align*}
    D' &\rightarrow D'\\
    \alpha &\mapsto \lsup{\sigma^{-1}}{(\mu_\sigma^{-1}\circ \alpha \circ \mu_\sigma)},
  \end{align*}
  as the condition on $\mu_\sigma$ shows that if $\alpha\in Z (D')$, then
  $ \lsup{\sigma^{-1}}{(\mu_\sigma^{-1}\circ \alpha \circ \mu_\sigma)} = \alpha$. By Skolem-Noether
  Theorem, there exists an element $\beta\in D'$ such that
  $\lsup{\sigma^{-1}}{(\mu_\sigma^{-1}\circ \alpha \circ \mu_\sigma)} = \beta\circ \alpha \circ
  \beta^{-1}$ for all $\alpha\in D'$. Thus
  $\alpha\circ \mu_\sigma \circ \lsup{\sigma}{\beta} = \mu_\sigma\circ \lsup{\sigma}{\beta} \circ
  \lsup{\sigma}{\alpha} $ for all $\alpha\in D'$.  Changing the isogeny $\mu_\sigma$ to $\mu_\sigma\circ \lsup{\sigma}{\beta}$, we get $D'$-equivariance.
\end{proof}
Noting how non-CM Abelian varieties of $\GL_2$-type decomposes (Prop.~\ref{prop:GL2-av-decomp}), we get the following:
\begin{cor}\label{cor:simple-factor-of-GL2AV-is-k-virtual-GL2AV}
  Absolutely simple factors of non-CM Abelian varieties of $\GL_2$-type over $k$ are $k$-virtual Abelian varieties of $\GL_2$-type.
\end{cor}

Given a simple $k$-virtual Abelian variety $B$ over $\bar {k}$
of $\GL_2$-type, one can construct a simple Abelian variety $A$ of $\GL_2$-type over $k$ such
that it has $B$ as an absolutely simple factor. This converse problem has been studied in \cite{MR1212980,MR2058652} over $\QQ$ and in \cite{MR2916969} over arbitrary number field $k$ even though the definition of $\GL_2$-type is more restrictive than here. Their methods generalise easily to the current case. Thus we just record the result.

\begin{prop}\label{prop:virtual-non-virtual-relation}
  Let $B$ be a non-CM $k$-virtual Abelian variety over $\bar {k}$ of $\GL_2$-type. Then there exists
  a non-CM simple Abelian variety $A$ over $k$ of $\GL_2$-type such that $A_{\bar {k}}$ is isogenous
  to a power of $B$.
\end{prop}

\section{Moduli Space of Virtual Abelian Varieties}
\label{sec:moduli}
The aim of this section is to determine a moduli space of $k$-virtual Abelian varieties of
$\GL_2 (E)$-type up to isogeny. One key step is the construction of $\lambda$-local trees (in the
sense of graph theory) for our Abelian varieties where $\lambda$ is a finite place of
$K$ where $K$ denotes the centre of $D=\End^0 (A)$.  Our construction generalises that of
Elkies \cite{MR2058644} where he associated certain trees to non-CM elliptic curves. The major
difficulty in the case of Abelian varieties comes from the fact that the endomorphism ring is much
more complicated.  We still manage to produce trees whose vertices are $k$-virtual Abelian varieties
of $\GL_2 (E)$-type up to a certain equivalence relation and whose edges represent simple
isogenies. Via graph theoretic properties of the trees, for a given $k$-virtual Abelian Varieties of
$\GL_2 (E)$-type, we can find an isogenous Abelian variety whose Galois orbit is contained in the
(generalised) Atkin-Lehner orbit. This makes it possible to represent $k$-virtual Abelian Varieties
by $k$-points on a quotient of a certain Shimura variety.

\subsection{Local Trees}
\label{sec:local-trees}

After excluding the CM case, there are two cases left for the endomorphism algebra of an Abelian
variety of $\GL_2 (E)$-type. One is when $\End^0 (A)$ is isomorphic to exactly $E$ and the other is
when $\End^0 (A)$ is isomorphic to a division quaternion algebra $D$ that contains $E$
(c.f. Prop.~\ref{prop:endom-alg-of-simple-av-GL2}).  After changing $A$ to an isogenous Abelian
variety, we may assume that $\End (A)$ is isomorphic to $\cO_E$ in the former case and that it is
isomorphic to a maximal order of $D$ in the latter case. Let $K$ be the centre of $D$. Fix a maximal
order $\cO_D$ of $D$. We use extensively results on maximal orders over complete discrete valuation
ring or over Dedekind domain. One good reference is Reiner's book\cite{MR0393100}. To unify the
construction for the two cases of endomorphism algebras, set
\begin{align*}
  &S = \cO_E \quad \text { and } \quad R =\cO_E , \\
  \text { or }\quad &S = \cO_D \quad \text { and } \quad R = \cO_K.
\end{align*}
Let $\cA (S)$ be the category where the objects are Abelian varieties $A$ of $\GL_2$-type such that
$\End (A)\isom S$ and the morphisms are $S$-linear isogenies. As usual, let $T_\ell A$ denote the
Tate-$\ell$-module associated to an Abelian variety $A$.  Let $\lambda$ be a prime of $R$. Write
$R_\ell$ for $R \otimes_{\ZZ} \ZZ_\ell$, $R_\lambda$ for the completion of $R$ at $\lambda$ and
$\cO_{D,\lambda}$ for $\cO_D\otimes_{\cO_K} \cO_{K,\lambda}$. Let $\varpi_\lambda$ be a uniformiser
of $\lambda$. To avoid confusion, sometimes we write $\varpi_{E,\lambda}$
(resp. $\varpi_{K,\lambda}$) to indicate which field we are working with. When $\lambda$ ramifies in
$D$, set $\varpi_{D,\lambda}$ be a uniformiser of the prime ideal of $\cO_{D,\lambda}$, i.e.,
$\varpi_{D,\lambda}^2 = u \varpi_{K,\lambda}$ for some $u\in \cO_{D,\lambda}^\times$.  Set
$T_\lambda A = T_\ell A \otimes_{R_\ell} R_\lambda$. This is a free $\cO_{E,\lambda}$-module
(resp. free left-$\cO_{D,\lambda}$-module) of rank $2$ (resp. $1$).

We construct a graph out of $\cA (S)$ as follows. Fix a prime $\lambda$ of $R$. The vertices are
equivalence classes of Abelian varieties in $\cA (S)$. We say the Abelian varieties $A$ and $B$ are
equivalent if there exists a morphism $f:A\rightarrow B$ such that the image of
$ T_\lambda f : T_\lambda A \rightarrow T_\lambda B$ is $\varpi_\lambda^n T_\lambda B$ for some
$n\in\Z_{\ge 0}$. Write $[A]_\lambda$ for the equivalence class of $A$.  Let $d=\deg f$. Then we
have a morphism  $g:B\rightarrow A$ such that the following diagram commutes:
\begin{equation*}
  \begin{tikzcd}[column sep=small]
    A\ar [rr,"f"] \ar [dr,"{[d]_A}"'] & & B \ar [dl,"g"]\\
    &A &
  \end{tikzcd}.
\end{equation*}
This induces the commutative diagram for Tate modules:
\begin{equation*}
  \begin{tikzcd}[column sep=small]
    T_\lambda A\ar [rr,"T_\lambda f"] \ar [dr,"{d}"'] & & T_\lambda B \ar [dl,"T_\lambda g"]\\
    &T_\lambda A &
  \end{tikzcd}.
\end{equation*}
The image of $d$ is of the form $\varpi_\lambda^{n'} T_\lambda A$ for some integer $n' \ge 0$. Thus
$T_\lambda g$ has image $\varpi_{\lambda}^{n'-n}T_\lambda A$, which shows that the equivalence relation is
well-defined.  We note that the Abelian varieties $A/A [\fa]$ for $\fa$ running over all ideals of
$R$ correspond to the same vertex in the graph for each $\lambda$.

Next we define the edges of the graph.  For $r\in\ZZ_{>0}$, set $M_r$ to be
\begin{enumerate}
\item (Case $E$) $\cO_{E,\lambda}/\varpi_{E,\lambda}^r\cO_{E,\lambda}$ if $S=\cO_E$;
\item (Case $D_\ram$) $\cO_{D,\lambda}/ \varpi_{D,\lambda}\cO_{D,\lambda}$, if $S=\cO_D$ and
  $\lambda$ ramifies in $D$;
\item (Case $D_\splt$) $(\cO_{K,\lambda}/ \varpi_{K,\lambda}^r\cO_{K,\lambda})^2$, if $S=\cO_D$ and
  $\lambda$ splits in $D$.
\end{enumerate}
When $r=1$, then we are quotienting out the maximal left ideal in all three cases.  In case
$D_\ram$, $r$ can only take the value $1$.  Connect two vertices $[A]_\lambda$ and $[B]_\lambda$ if
there exist representatives $A$ for $[A]_\lambda$ and $B$ for $[B]_\lambda$ and a morphism
$f:A\rightarrow B$ such that $\coker (T_\lambda f)$ is isomorphic to $M_1$.  We show that the edge
is bidirectional.  Let $d$ be the degree of $f$. Then we get a morphism $g: B\rightarrow A$ such
that $[d]_A = g\circ f$. Then the image of $d$ is of the form $\varpi_{\lambda}^{n'} T_\lambda A$
for some integer $n' > 0$. Assume we are in Case $E$.  Then for some choice of
$\cO_{E,\lambda}$-basis $\{e_1,e_2\}$ of $T_\lambda A$, the image of $T_\lambda g$ is
$\varpi_{E,\lambda}^{n'-1}\cO_{E,\lambda} e_1 \oplus \varpi_{E,\lambda}^{n'}\cO_{E,\lambda}
e_2$. Thus $g$ factors as the composition of the natural projection from $B$ to
$B/B [\lambda^{n'-1}]$ and a morphism $g': B/B[\lambda^{n'-1}] \rightarrow A$ with
$\coker (T_\lambda g') \isom \cO_{E,\lambda}/\lambda$. Assume that we are in Case $D_\ram$. Then the
image of $T_\lambda g$ is $\varpi_{D,\lambda}\varpi_{K,\lambda}^{n'-1}T_\lambda A$. Thus $g$ factors
as the composition of the natural projection from $B$ to $B/B [\lambda^{n'-1}]$ and a morphism
$g':B/B [\lambda^{n'-1}]\rightarrow A$ with
$\coker (T_\lambda g') \isom \cO_{D,\lambda}/\varpi_{D,\lambda}\cO_{D,\lambda}$. Assume that we are
in Case $D_\splt$. Fix an isomorphism $\cO_{D,\lambda}\isom \Mat_2 (\cO_{K,\lambda})$.  Then for
some choice of generators $e_1, e_2\in T_\lambda A$, image of $T_\lambda g$ is
$\varpi_{K,\lambda}^{n'-1} \Mat_2 (\cO_{K,\lambda}) e_1 \oplus \varpi_{K,\lambda}^{n'}\Mat_2
(\cO_{K,\lambda}) e_2$.  Thus $g$ factors as the composition of the natural projection from $B$ to
$B/B [\lambda^{n'-1}]$ and a morphism $g':B/B [\lambda^{n'-1}]\rightarrow A$ with
$\coker (T_\lambda g') \isom (\cO_{K,\lambda}/\varpi_{K,\lambda}\cO_{K,\lambda})^2$.

\begin{lemma}\label{lemma:no-loop-V-to-V}
  Assume that Abelian varieties $A$ and $B$ represent the same vertex. Then there does not exist a
  morphism $f:A\rightarrow B$ with $\coker (T_\lambda f)$ isomorphic to $M_1$. In particular, there
  is no loop in the graph.
\end{lemma}
\begin{proof}
  Assume that there exists such an $f$. Since $A$ and $B$ represent the same vertex, there exists a morphism $g:B\rightarrow A$ such that
  $T_\lambda g$ has image $\varpi_\lambda^n T_\lambda A$ for some $n\in \ZZ_{\ge 0}$. As the
  composite $g\circ f$ lies in the centre of $\End (A)$, $T_\lambda g \circ T_\lambda f$ has image
  of the form $\varpi_\lambda^{n'} T_\lambda A$ for some $n' \in \ZZ_{\ge 0}$. This is not
  possible if $f$ is such that $\coker (T_\lambda f)$ is isomorphic to $M_1$.
\end{proof}

At this point we see that in case $D_\ram$ each connected component of the graph consists of $2$
vertices connected by an edge and hence is a tree. Recall that in graph theory, a walk is an
alternating sequence of vertices and edges and a path is a walk in which all edges are distinct and
all vertices (except possibly the first and the last) are distinct.

\begin{lemma}\label{lemma:seq-of-morphisms-with-kernel-n}
  Exclude Case $D_\ram$.  Assume that two vertices $V_0$ and $V_n$ can be connected by a path of
  length $n$ via vertices $V_1$, $V_2$,... and $V_{n-1}$. Then there exists representatives $A_i$
  for $V_i$ for $i=0,\ldots ,n $ and morphisms $f_i:A_{i-1}\rightarrow A_i$ for $i=1,\ldots ,n $
  such that $\coker (T_\lambda f_i) \isom M_1$ for $i=1,\ldots ,n $ and that
  $\coker (T_\lambda (f_n \circ \cdots \circ f_1 )) \isom M_n$.
\end{lemma}
\begin{proof}
  When $n=1$, the statement is true by construction. Assume the statement holds for all paths with
  lengths less than $n$. We get a morphism
  \begin{equation*}
    g: A_1\xrightarrow {f_2} \cdots \xrightarrow {f_n} A_n
  \end{equation*}
  with $\coker (f_i)\isom M_1$ for $i=2,\ldots, n$ and $\coker (T_\lambda g) \isom M_{n-1}$ and a
  morphism $f_1: A_0\rightarrow A_1$ with $\coker (T_\lambda f)$ isomorphic to $M_1$. Consider the
  cokernel of the composite $T_\lambda (g\circ f_1)$. We separate the two cases.

  Assume that we are in Case $E$. Then $\coker (T_\lambda (g\circ f_1))$ is an extension of
  $\coker (T_\lambda g) \isom \cO_{E,\lambda}/\varpi_{E,\lambda}^{n-1}\cO_{E,\lambda}$ by
  $\coker (T_\lambda f)\isom \cO_{E,\lambda}/\varpi_{E,\lambda}\cO_{E,\lambda}$ as
  $\cO_{E,\lambda}$-modules. Thus it is isomorphic to
  \begin{equation*}
    \cO_{E,\lambda}/\varpi_{E,\lambda}^{n-1}\cO_{E,\lambda}\oplus \cO_{E,\lambda}/\varpi_{E,\lambda}\cO_{E,\lambda} \quad \text {or}\quad \cO_{E,\lambda}/\varpi_{E,\lambda}^n\cO_{E,\lambda}.
  \end{equation*}
  In the former case, we can see that $\coker (T_\lambda (f_2\circ f_1))$ must be isomorphic to
  $(\cO_{E,\lambda}/\varpi_{E,\lambda}\cO_{E,\lambda})^2$. This means that $A_0$ and $A_2$ represent
  the same vertex.

  Assume that we are in Case $D_\splt$.  Then $\coker (T_\lambda (g\circ f_1))$ is an extension of
  $\coker (T_\lambda g) \isom (\cO_{K,\lambda}/\varpi_{K,\lambda}^{n-1}\cO_{K,\lambda})^2$ by
  $\coker (T_\lambda f)\isom (\cO_{K,\lambda}/\varpi_{K,\lambda}\cO_{K,\lambda})^2$ as left
  $\Mat_2 (\cO_{K,\lambda})$-modules. If we consider extensions as $\cO_{K,\lambda}$-modules, then
  there are $3$ possibilities:
  \begin{align*}
    (\cO_{K,\lambda}/\varpi_{K,\lambda}\cO_{K,\lambda} \oplus \cO_{K,\lambda}/\varpi_{K,\lambda}^{n-1}\cO_{K,\lambda})^2, (\cO_{K,\lambda}/\varpi_{K,\lambda}^n\cO_{K,\lambda})^2,\\
    (\cO_{K,\lambda}/\varpi_{K,\lambda}\cO_{K,\lambda} \oplus \cO_{K,\lambda}/\varpi_{K,\lambda}^{n-1}\cO_{K,\lambda}) \oplus \cO_{K,\lambda}/\varpi_{K,\lambda}^n\cO_{K,\lambda}.
  \end{align*}
  The third one does not occur in the list of left $\Mat_2 (\cO_{K,\lambda})$-module
  extensions. Similar to the arguments in Case $E$, the first one will force $A_0$ and $A_2$ to
  represent the same vertex, leading to a contradiction.

  Thus in both cases we are led to the conclusion that $\coker (T_\lambda (g\circ f_1)) \isom M_n$.
\end{proof}

\begin{prop}
  Each connected component of the graph is a tree.
\end{prop}
\begin{proof}
  We need to show that there is no cycle. Assume that there is a cycle of length $n$ from the vertex
  $V$ to $V$. By Lemma~\ref{lemma:seq-of-morphisms-with-kernel-n}, there exist $A$ and $B$ in the
  equivalence class $V$ and a morphism $f:A\rightarrow B$ such that $\coker (T_\lambda f)\isom M_n$.
  This is not possible by Lemma~\ref{lemma:no-loop-V-to-V}. Thus there can be no cycle.
\end{proof}

\begin{defn}
  Let $\lambda$ be a prime of $R$.  Each connected component of the graph constructed above is
  called a $\lambda$-local tree. Let $A$ be an Abelian variety in $\cA (S)$. The $\lambda$-local
  tree containing the vertex $[A]_\lambda$ is called the $\lambda$-local tree associated to $A$.
\end{defn}

If $A\in \cA (S)$ is a $k$-virtual Abelian variety, then every vertex of the $\lambda$-local tree
associated to $A$ is an equivalence class of a $k$-virtual Abelian variety and the Galois group
$\Gal (\bar {k}/k)$ acts on the tree.

\begin{defn}
  Let $A$ be a $k$-virtual Abelian variety in $\cA (S)$. Set $O_\lambda (A)$ to be the
  $\Gal (\bar {k}/k)$-orbit $\{[\lsup{\sigma}{A}]_\lambda | \sigma \in \Gal (\bar {k}/k)\}$
  contained in the $\lambda$-local tree.
\end{defn}

A priori, the Galois orbit $O_\lambda (A)$ associated to a $k$-virtual Abelian variety $A$ is hard
to describe. However for some special vertices in the tree the Galois orbit is essentially contained
in the Atkin-Lehner orbit which we will describe below. We recall a definition from graph theory.

\begin{defn}
  For a finite subset $U$ of vertices of a tree, the centre of $U$ is defined to be the central edge or
  central vertex on any one of the longest paths connecting two vertices in $U$.
\end{defn}
\begin{rmk}
  There are possibly multiple longest paths, but they give the same centre. Thus the centre is
  well-defined.
\end{rmk}

\begin{defn}
  The $\lambda$-centre associated to a $k$-virtual Abelian variety $A\in \cA (S)$ is defined to be
  the centre of the Galois orbit $O_\lambda (A)$ in the $\lambda$-local tree.
\end{defn}
Since the Galois orbit $O_\lambda (A)$ is Galois stable, we have:
\begin{prop}\label{prop:centre-is-galois-fixed}
  The $\lambda$-centre of $k$-virtual Abelian variety $A\in \cA (S)$ is fixed under the action of
  $\Gal (\bar {k}/k)$. (If the centre is an edge it can possibly be flipped.)  Furthermore the
  vertices in the Galois orbit $O_\lambda (A)$ are at the same distance to (the nearer vertex of)
  the $\lambda$-centre.
\end{prop}
We consider the $\lambda$-centres that are central edges.
\begin{prop}
  The set of central edges associated to a $k$-virtual Abelian variety $A\in\cA (S)$ is an
  $S$-linear isogeny invariant. Thus it is an invariant for the $\lambda$-local tree.
\end{prop}
\begin{proof}
  Suppose the $\lambda$-centre associated to $A$ is an edge. Let $B\in\cA (S)$ be an Abelian variety
  that is $S$-linearly isogenous to $A$. We need to show that the $\lambda$-centre of $B$ is the same
  edge.
  
  First we note that there exists an element in $\Gal (\bar {k}/k)$ that exchanges the two vertices
  connected by the central edge. Otherwise all Galois conjugates of $A$ would be on one side of the
  edge, contrary to the fact that this edge is central.  Once we have an edge which is flipped under
  Galois action there can be no fixed vertices or other fixed edges in the tree.  Since the
  $\lambda$-centre associated to $B$ is fixed under Galois action, it must be the same edge that is
  the $\lambda$-centre for $A$.
\end{proof}

\begin{rmk}
  Central vertices are not necessarily isogeny invariants. For example we can take an Abelian
  variety $A\in\cA (\cO_E)$ over $k$ and take $B=A/C$ where $C$ is a $k$-subgroup of $A$ isomorphic
  to $\cO_E/\lambda$. Then obviously the central vertices, $[A]_\lambda$ and $[B]_\lambda$, are not
  the same vertex by construction.
\end{rmk}

\begin{defn}
  Let $A\in\cA (\cO_E)$ be a $k$-virtual Abelian variety. Set $\Sigma (A)$ to be the set of primes
  $\lambda$ of $\cO_E$ such that the $\lambda$-centre of $A$ is an edge.
\end{defn}
There is an analogous definition for $k$-virtual Abelian varieties in $\cA (\cO_D)$.  We note that
for each prime $\lambda$ of $\cO_K$, there exists a unique prime ideal $\til {\lambda}$ of $\cO_D$
that lies above $\lambda$\cite [Theorem~22.4] {MR0393100}. The $\lambda$-central edge associated to
$A$ determines a maximal left ideal $\cM_\lambda$ of $\cO_D$ that belongs to $\til {\lambda}$, in
the sense that $\til {\lambda}$ is the annihilator of $\cO_D/\cM_\lambda$ in $\cO_D$\cite
[Theorem~22.15] {MR0393100}.
\begin{defn}
  Let $A\in\cA (\cO_D)$ be a $k$-virtual Abelian variety. Set $\Sigma (A)$ to be the set of maximal
  left ideals $\cM_\lambda$ of $\cO_D$ determined by the $\lambda$-central edges of $A$.
\end{defn}
\begin{rmk}
  To unify the notation, we also write $\cM_\lambda$ for $\lambda$ in case $A \in \cA (\cO_E)$.
\end{rmk}

\begin{lemma}
  The set $\Sigma (A)$ is a finite set and for almost all $\lambda$'s, $[A]_\lambda$ is its own
  $\lambda$-centre.
\end{lemma}
\begin{proof}
  The Abelian varieties $\lsup{\sigma}{A}$ for $\sigma$ in $\Gal (\bar {k}/k)$ end up in the same
  equivalence class as $A$, as long as $\lambda$ does not divide the degree of the isogenies $\mu_\sigma$'s
  between the Galois conjugates. Thus there are only finitely many $\lambda$'s such that the
  $\lambda$-centre associated to $A$ can be an edge.
\end{proof}

For each $\cM_\lambda\in\Sigma (A)$ we choose one of the vertices $V_\lambda$ on the central edge
and for each $\cM_\lambda\notin \Sigma (A)$ we just use the central vertex $V_\lambda$.  The path
connecting $[A]_\lambda$ to $V_\lambda$ describes the `$\lambda$'-part of an isogeny. Thus the
chosen vertices give rise to an  Abelian variety $A_0$ isogenous to $A$. As the vertices are equivalence
classes of Abelian varieties, we cannot distinguish the Abelian varieties $A_0$ from $A_0/A_0 [\fa]$
for integral ideals $\fa$ of $R$, or in terms of Serre tensor (c.f. for example \cite [Sec.~1.7.4]
{MR3137398}), from $A_0\otimes_{R}\fa$ for fractional ideals $\fa$ of $R$. These Abelian varieties
are exactly the ones that correspond to the vertex $[A_0]_\lambda$ for each $\lambda$.  We note that
if $\fa$ is principal, then $A_0\otimes_{R}\fa$ is isomorphic to $A_0$. Thus the set of these
vertices $V_\lambda$ determines an Abelian variety up to the action of the class group $\Cl (R)$.

Now let $A_0$ be an Abelian variety such that $[A_0]_\lambda = V_\lambda$ for all prime $\lambda$ of
$R$. The finitely many central edges determine a level structure on $A_0$. This is an embedding
\begin{equation*}
  \oplus_{\cM_\lambda \in \Sigma (A_0)} S/\cM_\lambda \hookrightarrow A_0.
\end{equation*}
For $A_0 \in \cA (\cO_E)$, the left-hand side is isomorphic to
$\cO_E/ \cap_{\lambda \in \Sigma (A_0)}\lambda$ or $\cO_E/ \prod_{\lambda \in \Sigma (A_0)}\lambda$.
For $A_0 \in \cA (\cO_D)$, the left-hand side is isomorphic to
$\cO_D/ \cap_{\cM_\lambda \in \Sigma (A_0)}\cM_\lambda$, since the $\cM_\lambda$'s belong to
distinct primes of $\cO_K$. By Prop.~\ref{prop:centre-is-galois-fixed}, the Galois conjugates of $A_0$ must correspond to the central vertices or the vertices on the central edges. In other words, isogenies from the Galois conjugates to $A_0$ are controlled by the level structure. Thus we have shown:
\begin{thm}\label{thm:galois-orbit-bound-by-level-structure}
  For every $k$-virtual Abelian variety $A\in\cA (S)$, there exists a $k$-virtual Abelian variety
  $A_0\in\cA (S)$ which is $S$-linearly isogenous to $A$ and a level structure
  \begin{equation*}
    \eta: S/ \cap_{\cM_\lambda \in \Sigma (A_0)}\cM_\lambda \hookrightarrow A_0
  \end{equation*}
  such that for all $\sigma\in\Gal (\bar {k}/k)$, there exist some ideal $\cI$ of $R$ and  an $S$-linearly isogeny $A_0 \rightarrow \lsup{\sigma}{A_0}\otimes_{R} \cI$ with  kernel contained in the image of $\eta$.
\end{thm}

\subsection{Shimura Varieties of PEL Type}
\label{sec:shim-vari-pel}

The analysis in the previous subsection leads us to consider PEL Shimura varieties 
that classify Abelian varieties $A$ with endomorphism ring that contains $S$ and with level
structure $S/\cN\hookrightarrow A$ where $\cN$ is a full left ideal of $S$ that is square-free, in
the sense that $\cN_\lambda$ is either $S_\lambda$ or a maximal ideal of $S_\lambda$. We separate
the discussion into two cases.

\subsubsection{Abelian varieties with $\cO_E$-action}
\label{sec:abel-vari-with-OE-action}

Consider Abelian varieties of dimension $[E:\QQ]$ with $\cO_E$-action.  The moduli spaces of Abelian varieties of dimension $[E:\QQ]$ with
endomorphism algebra $E$ and a prescribed analytic representation of $E$ were studied by
Shimura\cite{MR0156001}.  We restrict to the case where $E$ is totally real. This is not a
restriction if we are in the situation (a) of
Prop.~\ref{prop:endom-alg-of-simple-av-GL2-with-analytic-rep-restriction-or-defined-over-R}. Then
the PEL Shimura varieties are Hilbert modular varieties that classify Abelian varieties $A$ with
real multiplication $E$ and with level structure $\cO_E/\fn\hookrightarrow A$ where $\fn$ is a
square-free ideal of $\cO_E$. We describe more precisely the moduli problem below.

For a
fractional ideal $\fa$ of $\cO_E$, let $\GL (\cO_E \oplus \fa)$ denote the subgroup of $\GL_2 (E)$
that stabilises the lattice $\cO_E \oplus \fa$ (with action on the right).  Let $A$ be an Abelian
variety such that there is an embedding $\iota:\cO_E\rightarrow \End (A)$. Let $\NS(A)$ denote the
N\'eron-Severi group of $A$. Let $t_x$ denote the translation by $x$ map for $A$. Then We have an
embedding
\begin{align*}
  \NS (A) &\rightarrow \Hom(A,\hat {A})\\
  \cL &\mapsto \phi_\cL: x \mapsto t^*_x\cL \otimes \cL^{-1}
\end{align*}
with image being the set of the symmetric elements in $\Hom (A,\hat {A})$. Set $\NS_E(A)$ to be the
set of $\cL\in \NS (A)$ such that $\phi_\cL \circ \alpha = \hat {\alpha}\circ \phi_\cL$ for all
$\alpha\in \cO_E$ with $\hat {\alpha}$ being the dual endomorphism $\hat {A}\rightarrow \hat
{A}$. The action of $\cO_E$ on $\Hom(A,\hat {A})$ induces an action of $\cO_E$ on $\NS_E (A)$ ,
making it into an $\cO_E$-module of rank $1$. In other words, $\NS_E (A)$ is isomorphic to a
fractional ideal of $\cO_E$.

Let $\fc$ run over a set of representatives of the narrow class group $\Cl^+ (E)$ of $E$.  For each
embedding $\iota: E\hookrightarrow \R$, we fix an ordering of $\fc \otimes_\iota \R$ and thus get a
notion of positivity on $\fc$. We consider Abelian variety $A$ of dimension $g := [E:\Q]$ with
\begin{itemize}
\item $\iota:\cO_E\hookrightarrow \End (A)$ such that the induced action of $E$ on $\Lie (A)_\CC$ is
  given by the $g$ embeddings of $E$ into $\CC$,
\item a weak polarisation $\NS_E (A) \xrightarrow {\isom} \fc$ that maps polarisations to positive
  elements in $\fc$,
\item a level structure $\eta:\cO_E/\fn\hookrightarrow A$.
\end{itemize}

The isomorphism classes of such complex Abelian varieties are parametrised by the complex points of
the Hilbert modular variety
\begin{equation*}
  Y_0 (\fn) (\CC) =\GL_2 (E) \lmod (\cH^\pm)^g \times \GL_2 (\A_{E,f}) / U_0 (\fn).
\end{equation*}
where the compact open subgroup $U_0 (\fn)$ of $\GL_2 (\A_{E,f})$ is defined to be the product of
$\GL (\cO_{E,\lambda}\oplus \cO_{E,\lambda}) \cap \GL (\cO_{E,\lambda}\oplus
\fn^{-1}\cO_{E,\lambda}) $ over all finite places $\lambda$ of $E$. Let $U_\infty$  be the
stabiliser of $(i,\ldots , i) \in \cH^{+}$. The determinant map
\begin{equation*}
  \GL_2 (E) \lmod  \GL_2 (\A_E) / U_\infty U_0 (\fn) \rightarrow E^\times\lmod \A_E^\times /  \A_{E,\infty}^+\prod_{\lambda \nmid \infty} \cO_{E,\lambda}^\times.
\end{equation*}
shows that $Y_0 (\fn) (\CC)$ has $\#\Cl^+ (E)$ connected components.
 
Now we bring the local trees into play. As $A\otimes_{\cO_E}\cI$'s for $\cI\in\Cl (E)$ correspond to the same
vertex in each $\lambda$-local tree, we need to consider the action of $\Cl (E)$ on $ Y_0 (\fn)$.  The action
of $\Cl (E)$ on the isogeny $A \rightarrow A'$ with $A,A'\in \cA (E)$ and kernel isomorphic to
$\cO_E/\fn$ is given by
\begin{equation*}
  (A\rightarrow A') \mapsto (A\otimes_{\cO_E}\cI \rightarrow A'\otimes_{\cO_E}\cI).
\end{equation*}
This changes the polarisation module $\fc$ to $\fc\cI^{-2}$. We also need to consider the flipping
of central edges in $\lambda$-local trees. For each $\lambda|\fn$ we get an action:
\begin{equation*}
  w_\lambda: (A\xrightarrow {f} A') \mapsto (A/\eta (\cO_E/\lambda)\rightarrow A'/f (A [\lambda])).
\end{equation*}
As another application of $w_\lambda$ gives $(A/A [\lambda]\rightarrow A'/A' [\lambda])$, we see
that it is an involution on $\Cl (E)\lmod Y_0 (\fn)$.

\begin{defn}
  \label{defn:extended-atkin-lehner}
  Define the extended Atkin-Lehner group to be the group generated by $\Cl (E)$ and $w_\lambda$'s
  for $\lambda|\fn$ and denote it by $\tilde {W}_{E,\fn}$. Set
  $Y_0^+ (\fn) = \tilde {W}_{E,\fn} \lmod Y_0 (\fn)$.
\end{defn}
With this, we can paraphrase Theorem~\ref{thm:galois-orbit-bound-by-level-structure} above:
\begin{thm}
  Let $A\in\cA (\cO_E)$ be a $k$-virtual Abelian variety. Then there exists an Abelian variety $A_0$
  that is $\cO_E$-linearly isogenous to $A$ such that the Galois orbit $\Gal (\bar {k}/k).A_0$ is
  contained in the extended Atkin-Lehner orbit $\tilde {W}_{E,\fn}.A_0$.
\end{thm}

Thus by construction, $A_0$ is a $k$-rational point on $Y_0^+(\fn)$.  On the other hand take a
$k$-rational point of $Y_0^+(\fn)$ and we get a set $\Sigma$ of Abelian varieties in $ Y_0 (\fn)$
that lie above it. They are isogenous to each other. Take any one of them, say $A_0$. Then its
$\Gal (\bar {k}/k)$-conjugates are still in the set $\Sigma$ and they are $\cO_E$-linearly isogenous
to $A_0$ by construction. This means that $A_0$ is a $k$-virtual Abelian variety of $\GL_2(E)$-type,
even though it may possibly have strictly larger endomorphism algebra than $E$.

We have shown
\begin{thm}
  \label{thm:moduli_of_virtual_av_E}
  Every $k$-point on the Hilbert modular variety $Y_0^+(\fn)$ gives rise to a $k$-virtual Abelian
  variety of $\GL_2(E)$-type. Conversely for any $k$-virtual Abelian variety $A$ of $\GL_2(E)$-type
  with endomorphism algebra isomorphic to $E$ there exists a $k$-virtual Abelian variety of
  $\GL_2(E)$-type $A_0$ that is $E$-linearly isogenous to $A$ such that it corresponds to a
  $k$-rational point on $Y_0^+(\fn (A))$ with $\fn (A) = \prod_{\lambda\in\Sigma (A)}\lambda$.
\end{thm}

\subsubsection{Abelian varieties with $\cO_D$-action}
\label{sec:abel-vari-with-OD-action}

Next consider Abelian varieties of dimension $[D:\QQ]/2$ with $\cO_D$-action where $\cO_D$ is a maximal order of $\cO_D$. In
the situation of
Prop.~\ref{prop:endom-alg-of-simple-av-GL2-with-analytic-rep-restriction-or-defined-over-R}, $D$
must be a totally indefinite quaternion algebra over a totally real field $K$. This is the case we
will pursue. We consider the PEL Shimura varieties that parametrise Abelian varieties with
$\cO_D$-action with level structure $\cO_D/\cN$ where $\cN$ is a full left ideal of $S$ that is
square-free, in the sense that $\cN_\lambda$ is either $\cO_{D,\lambda}$ or a maximal ideal of
$\cO_{D,\lambda}$. Fix a positive anti-involution $\dagger$ on $D$ that fixes $K$ element-wise. Let
$A$ be an Abelian variety with $\iota: \cO_D\rightarrow \End (A)$. We consider the subset
$\NS_D (A)$ of $\NS (A)$ that is compatible with $D$ along with the anti-involution. In other words,
$\NS_D (A)$ consists of $\cL\in \NS (A)$ such that
$\widehat {\iota (\alpha)} \circ \phi_\cL = \phi_\cL \circ \iota (\alpha^\dagger)$. As $\NS (A)$
embeds into the symmetric part of $\Hom (A,\hat {A})$, we see that $\NS_D (A)$ is an $\cO_K$-module
of rank $1$. Let $\fc$ run over a set of representatives of the narrow class group $\Cl^+ (K)$ of
$K$.  As in the previous case we have a notion of positivity on $\fc$. Thus we consider Abelian
varieties $A$ of dimension $g:= 2[K:\QQ]$ with
\begin{itemize}
\item $\iota: \cO_D \hookrightarrow \End (A)$ such that the action of $D$ on $\Lie (A)_\CC$ is the
  direct sum of the standard representation composed with the $g/2$ embedding of $D$ into
  $\Mat_2 (\CC)$ respectively;
\item a weak polarisation $\NS_D (A) \xrightarrow {\isom} \fc$ that maps polarisations to positive
  elements in $\fc$,
\item a level structure $\eta: \cO_D/\cN \rightarrow A$.
\end{itemize}

Let $G$ be the algebraic group over $K$ determined by $G (K) = D^\times$.  The isomorphism classes
of such Abelian varieties are parametrised by the $\CC$-points of the quaternionic Shimura variety
\begin{equation*}
  Sh_0 (\cN) (\CC):= G (K) \lmod (\cH^{\pm})^g \times G (\AA_{K,f})^\times / U_0 (\cN)
\end{equation*}
where $U_0 (\cN)$ is the product of $U_0 (\cN)_\lambda$ over finite places $\lambda$ of $K$ defined
as follows. When $\lambda$ ramifies in $D$, set $U_0 (\cN)_\lambda$ to be $\cO_{D,\lambda}^\times$;
when $\lambda$ splits in $D$ and $\cO_D/\lambda\cO_D$ is not a composition factor of $ \cO_D/\cN$,
set $U_0 (\cN)_\lambda$ to be $\GL_2 (\cO_{K,\lambda})$; when $\lambda$ splits in $D$ and
$\cO_D/\lambda\cO_D$ is a composition factor of $\cO_D/\cN$, set $U_0 (\cN)_\lambda$ to be subgroup
of $\GL_2 (\cO_{K,\lambda})$ with lower-left element in $\varpi_{K,\lambda}\cO_{K,\lambda}$. The
reduced norm map
\begin{equation*}
  G(K) \lmod  G (\A_K) / U_\infty U_0 (\cN) \rightarrow K^\times\lmod \A_K^\times /  \A_{K,\infty}^+\prod_{\lambda \nmid \infty} \cO_{K,\lambda}^\times.
\end{equation*}
is surjective by Eichler's theorem and this shows that $Sh_0 (\cN) (\CC)$ has $\#\Cl^+ (K)$
connected components.

Now we bring the local trees into play. As $A\otimes_{\cO_K}\cI$'s for $\cI\in\Cl (K)$ correspond to the same
vertex in each $\lambda$-local tree, we need to consider the action of $\Cl (K)$ on $ Sh_0 (\cN)$.  The action
of $\Cl (K)$ on the isogeny $A \rightarrow A'$ with kernel isomorphic to $\cO_E/\cN$ is given by
\begin{equation*}
  (A\rightarrow A') \mapsto (A\otimes_{\cO_K}\cI \rightarrow A'\otimes_{\cO_K}\cI).
\end{equation*}
This changes the polarisation module $\fc$ to $\fc\cI^{-4}$. We also need to consider the flipping
of central edges in $\lambda$-local trees.
For each $\lambda$ in the support of $\cO_D/\cN$, we get an action:
\begin{equation*}
  w_\lambda: (A\xrightarrow {f} A') \mapsto (A/\eta (\cO_D/\cM_\lambda)\rightarrow A'/f (A [\lambda])).
\end{equation*}
where $\cM_\lambda$ is the `$\lambda$-part' of $\cN$ such that
$\cO_D/\cM_\lambda \isom \cO_{D,\lambda}/\cN_\lambda$.  As another application of $w_\lambda$ gives
$(A/A [\lambda]\rightarrow A'/A' [\lambda])$, we see that it is an involution on
$\Cl (K)\lmod Sh_0 (\cN)$.

\begin{defn}
  \label{defn:extended-atkin-lehner-D}
  Define the extended Atkin-Lehner group to be the group generated by $\Cl (K)$ and $w_\lambda$'s
  for $\lambda$ in the support of $\cO_D/\cN$ and denote it by $\tilde {W}_\cN$. Set
  $Sh^+ (\cN) = \tilde {W}_\cN \lmod Sh_0 (\cN)$.
\end{defn}
The analogous theorems for the quaternionic case are as follows.

\begin{thm}
  Let $A\in\cA (\cO_D)$ be a $k$-virtual Abelian variety. Then there exists an Abelian variety $A_0$
  that is $\cO_D$-linearly isogenous to $A$ such that the Galois orbit $\Gal (\bar {k}/k).A_0$ is
  contained in the extended Atkin-Lehner orbit $\tilde {W}_\cN.A_0$.
\end{thm}

\begin{thm}
  \label{thm:moduli_of_virtual_av_D}
  Every $k$-point on the quaternionic Shimura variety $Sh^+(\cN)$ gives rise to a $k$-virtual
  Abelian variety of $\GL_1 (D)$-type. Conversely for any $k$-virtual Abelian variety $A$ of
  $\GL_1 (D)$-type with endomorphism algebra isomorphic to $D$, there exists a $k$-virtual Abelian
  variety $A_0$ of $\GL_1(D)$-type that is $D$-linearly isogenous to $A$ such that it corresponds to
  a $k$-rational point on $Sh^+(\cN (A))$ with $\cN (A)=\cap_{\cM \in \Sigma (A)} \cM$.
\end{thm}

\section{Classification of Hilbert Modular Surfaces}
\label{sec:class-hilb-modul}
We would like to apply the Enriques-Kodaira classification to our moduli spaces of $k$-virtual Abelian varieties when they are Hilbert modular surfaces. The main
reference is van der Geer's book\cite{MR930101}. See also the many works of Hirzebruch and his joint
works with Van de Ven or Zagier on Hilbert modular surfaces that date before it,
for example, \cite{MR0393045,MR0364262,MR0480356}. According to Lang's conjecture, we do not expect
to have many rational points on varieties of general type. Thus such classification will give us
some rough idea where $k$-virtual Abelian varieties can be found in more abundance. More detailed
analysis of the Hilbert modular varieties and the quaternionic Shimura varieties will be part of our future work.

In this section, we will focus on the case where $E$ is a real quadratic field with narrow class
number $1$ and study the Hilbert modular surfaces $Y_0^+(\fp)$ where $\fp$ is a prime ideal of
$E$. This assumption is always in effect. To avoid too much repetition, we will omit it from our
statements. We keep the notation from Sec.~\ref{sec:moduli}.

First we note some implications of the assumption that $|\Cl^+(E)| =1$. Suppose $E=\Q(\sqrt{D})$
where $D$ is the discriminant.  Then $D$ is necessarily either a prime congruent to $1$ modulo $4$
or $D=8$. The torsion-free part of the group of units $\cO_E^\times$ is generated by an element with
norm equal to $-1$. Thus in our case $\PSL_2(\cO_E)=\PGL_2^+(\cO_E)$. Let $\Gamma_0^E (\fp)$ denote
the subgroup of elements in $\PSL_2 (\cO_E)$ whose lower-left entry is congruent to $0$ modulo
$\fp$.    Since, a
fortiori, the class group of $E$ is trivial, the group $\tilde{W}_{E,\fp}$ in
Definition~\ref{defn:extended-atkin-lehner} is a group of order $2$.  For the sake of brevity, we
will denote it by $W$. More precisely, let $\fp= (\varpi_\fp)$ with $\varpi_\fp$ chosen to be
totally positive. Then $W$ is generated by the involution on $\cH^2$ given by the action of the
element $w_\fp=\bigl( \begin{smallmatrix}
  0 &1\\
  -\varpi_\fp &0
\end{smallmatrix} \bigr)$. Hence $Y_0^+(\fp)$ is isomorphic to $W\Gamma_0^E(\fp) \lmod \cH^2$.

Let $\bar {Y}_0^+ (\fp)$ be the
compactification of $Y_0^+ (\fp)$ which is given by $W\Gamma_0^E(\fp) \lmod \cH^2 \cup \PP^1(E)$.  Let $X_0^+(\fp)$ denote the minimal desingularisation of
$\bar {Y}_0^+ (\fp)$. The book of van der Geer\cite{MR930101} on Hilbert modular surfaces does not
consider level structure, so it does not cover our case. However we do rely heavily on its
techniques. We are able to show that most surfaces in question are of general type. We will also
give some examples of surfaces that are not of general type. First we review how one resolves 
singularities on the surfaces and then estimate the Chern numbers.

\subsection{Cusp Singularities}
\label{sec:cusp_sing}

For $\Gamma_0^E(\fp)\lmod \uhp^2 \cup \PP^1(E)$ there are two inequivalent cusps $0$ and
$\infty$. They are identified via the Atkin-Lehner operator $w_\fp$. The isotropy group of the
unique inequivalent cusp $\infty$ in $W\Gamma_0^E(\fp)$ is equal to that in $\PSL_2(\cO_E)$, as
$W\Gamma_0^E(\fp)$ contains all of those elements in $\PSL_2(\cO_E)$ that are of the form
$\bigl( \begin{smallmatrix}
  a &b\\
  0 &d
\end{smallmatrix} \bigr)$.  Thus the type of the cusp singularity is the same as that for
$\PSL_2(\cO_E)$ and the isotropy group is equal to
\begin{equation}
  \begin{split}
    &\left\lbrace
      \begin{pmatrix}
        \epsilon & \mu\\
        0 &\epsilon^{-1}
      \end{pmatrix}
      \in \SL_2(E) : \epsilon \in \cO_E^{\times},\mu\in \cO_E
    \right\rbrace / \{\pm I\}\\
    \isom&\left\lbrace
      \begin{pmatrix}
        \epsilon & \mu\\
        0 &1
      \end{pmatrix}
      \in \GL_2^+(E) : \epsilon \in \cO_E^{\times +},\mu\in \cO_E
    \right\rbrace\\
    \isom &\ \cO_E \rtimes \cO_E^{\times +}.
  \end{split}
\end{equation}
By \cite[Chapter II]{MR930101} we have the minimal resolution of singularity resulting from toroidal
embedding and the exceptional divisor consists of a cycle of $\PP^1$'s or of one rational curve with
one ordinary double point.

\begin{defn}
  Let $C_1,\ldots , C_m$ be rational curves on a non-singular surface and let $b_1, \ldots , b_m$ be
  integers. If $m\ge 2$, we require that $C_1,\ldots , C_m$ are non-singular. If $m=1$, we require
  that $C_1$ is a rational curve with one ordinary double point. Set $C_0=C_m$. We say
  $C_1,\ldots , C_m$ form a cycle of type $[b_1, \ldots , b_m]^\circ$ if the following hold.
  \begin{enumerate}
  \item When $m\ge 3$, the intersection number $C_i.C_j$ is equal to $1$, if $|i-j|=1$, to $-b_i$ if
    $i=j$ and to $0$ otherwise;
  \item When $m=2$, the intersection number $C_i.C_j$ is equal to $2$, if $|i-j|=1$, to $-b_i$ for
    $i=j$;
  \item When $m=1$, the intersection number $C_1.C_1$ is equal to $-b_1 +2$.
  \end{enumerate}
\end{defn}

\subsection{Elliptic Fixed Points}
\label{sec:ell_pt}
Now consider the inequivalent elliptic fixed points of $W\Gamma_0^E(\fp)$ on $\cH^2$. More generally
we consider the elliptic fixed points of $\PGL_2^+(E)$. Suppose $z=(z_1,z_2)$ is fixed by
$\alpha=(\alpha_1,\alpha_2)$ in the image of $\PGL_2^+(E)$ in $\PGL_2^+(\R)^2$. Then
\begin{equation*}
  \alpha_j
  .z_j = \frac{a_jz_j + b_j}{c_jz_j +d_j}=z_j
\end{equation*}
for $j=1$ or $2$ where $\alpha_j=\bigl(
\begin{smallmatrix}
  a_j &b_j\\c_j&d_j
\end{smallmatrix}\bigr)$. Solving the equation we get
\begin{equation}\label{eq:ell-pt-coords}
  z_j= \frac{a_j-d_j}{2c_j}+ \frac{1}{2|c_j|} \sqrt{\tr^2 (\alpha_j) -4\det(\alpha_j)}.
\end{equation}
Transform $z_j$ to $0$ via the m\"obius transformation
$\zeta_j \mapsto \frac{\zeta_j-z_j}{\zeta_j-\bar{z}_j}$ of $\CC$. Then the isotropy group of $z_j$
acts as rotation around $0$ on each factor $\CC$. The action of $\alpha_j$ transfers to that of
$\gamma_j\alpha_j\gamma_j^{-1}$ where $\gamma_j = \smatrix{1}{-z_j}{1}{-\bar {z_j}}$. A little
computation shows that
\begin{equation*}
  \gamma_j\alpha_j\gamma_j^{-1} = (z_j - \bar {z}_j)^{-1}
  \begin{pmatrix}
    -a_j\bar {z}_j + cz_j\bar {z}_j + d_jz_j - b_j & 0\\
    0 & a_j z_j - c_jz_j\bar {z}_j - d_j\bar {z}_j + b_j
  \end{pmatrix}.
\end{equation*}
Using the equation that $z_j$ satisfies we get that the above is equal to
\begin{equation*}
  (z_j - \bar {z}_j)^{-1}
  \begin{pmatrix}
    (a_j-c_jz_j) (z_j - \bar {z}_j) & 0\\
    0 & (a_j-c_j\bar {z}_j) (z_j - \bar {z}_j)
  \end{pmatrix}.
\end{equation*}
Thus the rotation angle is twice the argument $\theta_j$ of $a_j-c_jz_j$ which satisfies
\begin{equation}
  \label{eq:rotation}
  \cos\theta_j=\frac{\tr(\alpha_j)}{2\sqrt{\det(\alpha_j)}}, \quad
  c_j\sin\theta_j <0.
\end{equation}
The isotropy group of an elliptic point is cyclic.
\begin{defn}
  We say that the quotient singularity at $(z_1,z_2)\in\cH^2$ is of type $(n;a,b)$ if after
  transferring $(z_1,z_2)$ to $(0,0)$ as above, a generator of the isotropy group acts as
  $(w_1,w_2)\mapsto (\zeta_n^a w_1, \zeta_n^b w_2)$ where $\zeta_n$ is a primitive $n$-th root of
  $1$.
\end{defn}
\begin{rmk}
  Of course, some types are equivalent. We may require that at least one of $a$ and $b$ is coprime
  to $n$.  When $a$ is coprime to $n$, we may require that $a$ is equal to $1$ by changing the
  chosen primitive $n$-th root of $1$. In fact, in the situation we encounter later, both $a$ and
  $b$ will be coprime to $n$. Then the quotient singularity is an isolated singularity. It is shown
  in \cite[Section 6, Chapter II]{MR930101} that the exceptional divisor in the resolution of cyclic
  quotient singularity is a chain of $\PP^1$'s. The table on page~65 of \cite{MR930101} gives some
  explicit examples.
\end{rmk}
\begin{defn}
  Let $C_1,\ldots , C_d$ be non-singular rational curves on a surface $S$. Assume that $C_i^2=-c_i$
  for $i=1,\ldots , d$, $C_{i-1}.C_i=1$ for $i=2,\ldots , d$ and that the rest of the intersection
  numbers involving these non-singular rational curves are $0$. Then we say that $C_1,\ldots , C_d$
  form a chain of type $[c_1,\ldots , c_d]$.
\end{defn}
\begin{rmk}\label{rmk:quot-resol}
  With this definition we can be more precise about the exceptional divisor coming from the
  resolution of cyclic quotient singularity. For cyclic quotient singularity of type $(n;1,1)$, the
  exceptional divisor is of type $[n]$; for quotient singularity of type $(n;1,-1)$, the exceptional
  divisor is of type $[2,\ldots , 2]$ where $2$ appears $n-1$ times.
\end{rmk}

\begin{defn}
  Let $\Gamma$ be a discrete subgroup of $\PGL_2^+ (\R)^2$. Let $a_n^{\pm} (\Gamma)$ denote the
  number of $\Gamma$-inequivalent elliptic points of type $(n;1,\pm 1)$.  When $n=2$, simply set
  $a_2(\Gamma) = a_2^+ (\Gamma) =a_2^-(\Gamma)$.
\end{defn}
Now we restrict to the elliptic fixed points of $\Gamma_0^E (\fp)$. For there to be any, we need
$\tr^2 (\alpha_j) < 4 \det (\alpha_j) =4$ for $j=1,2$ for some $\alpha\in\Gamma_0^E (\fp)$. For
varying discriminant $D$ of the real quadratic field $E$, the only possible values $\tr (\alpha_j)$
can assume are:
\begin{equation*}
  0, \pm 1, \pm \sqrt {2}, \pm \sqrt {3}, (1\pm \sqrt {5})/2.
\end{equation*}
Then from the expression \eqref{eq:rotation} for $\cos (\theta_j)$, we get:
\begin{lemma}
  When the discriminant $D$ is greater than $12$,  the elliptic elements of $\Gamma_0^E(\fp)$ can
  only be of order $2$ or $3$.
\end{lemma}

\subsection{Estimation of Chern Numbers}
\label{sec:Chern_num}

Let $X_0 (\fp)$ be the minimal desingularisation of $\Gamma_0^E(\fp) \lmod \cH^2 \cup
\PP^1(E)$. Then it is simply-connected since there is no non-trivial Hilbert modular form of weight
$(2,0)$ or $(0,2)$\cite [Lemma~6.3] {MR930101}. The Atkin-Lehner operator $w$ extends to an
involution on $X_0 (\fp)$ which has at least one fixed point. Thus we see that the quotient
 $X_0^+ (\fp)$ is simply-connected. Equivalently $X_0^+ (\fp)$ is
a surface with vanishing irregularity. We rely on the table \cite [Table~10, Page 244] {MR2030225}
which gives the Enriques-Kodaira classification for minimal surfaces. For easier reference, we record
in Table~\ref{table:kodaira} the rows where the first Betti number $b_1$ can possibly be zero. Let
$c_i$ be the $i$-th Chern class. The Chern class $c_i(S)$ of a surface $S$ is the Chern class of the
tangent bundle. Let $\chi$ denote the Euler characteristic and $p_a$ denote the arithmetic genus.
\begin{table}
  \centering
  \caption {Classification of Surfaces.\label{table:kodaira}}
  \begin{tabular}{|c|c|c|c|}
    \hline
    Class of Surface & Kodaira Dimension & $c_1^2$ & $c_2$\\
    \hline
    minimal rational surface & $-\infty$ & $8$ or $9$ & $4$ or $3$\\
    Enriques surface & $0$ & $0$ & $12$\\
    K3 surface & $0$ & $0$ & $24$\\
    minimal honestly elliptic surface & $1$ & $0$ & $\ge 0$\\
    minimal surface of general type & $2$ & $>0$ & $>0$\\
    \hline                                                
  \end{tabular}
\end{table}
It is not known if $X_0^+ (\fp)$ is a minimal surface. As blowing down an exceptional curve
increases $c_1^2 (X_0^+ (\fp))$ by $1$ and leaves
$\chi= (c_1^2 (X_0^+ (\fp)) + c_2 (X_0^+ (\fp)))/12$ invariant, we have the following criterion.

\begin{prop}
  \label{prop:criterion_general_type}
  Let $S$ be a nonsingular algebraic surface with vanishing irregularity. If $\chi > 1$ and
  $c_1^2 (S) >0$, then $S$ is of general type.
\end{prop}

Now we will estimate the Chern numbers of $X_0^+(\fp)$. We begin by defining the local Chern cycle.

\begin{defn}
  Let $S$ be a normal surface with isolated singular points and let $S'$ be its
  desingularisation. Suppose $p$ is a singular point on $S$ and the irreducible curves
  $C_1,\ldots,C_m$ on $S'$ form the resolution of $p$. Then the local Chern cycle of $p$ is defined
  to be the unique divisor $Z=\sum_{i=1}^m a_i C_i$ with rational numbers $a_i$ such that the
  adjunction formula holds:
  \begin{equation*}
    Z.C_i-C_i.C_i=2-2p_a(C_i).
  \end{equation*}
\end{defn}
\begin{rmk}
  \label{rmk:local-chern-cycle}
  We can be precise about what the exceptional divisors are for cyclic quotient singularities and
  cusp singularities.

  For cyclic quotient singularity of type $(n;1,1)$, the exceptional divisor is of type $[n]$ and
  consists of one non-singular rational curve $C_1$; the local Chern cycle is $(1-2/n)C_1$. For
  quotient singularity of type $(n;1,-1)$, the exceptional divisor consists of a chain of
  non-singular rational curves $C_1,\ldots , C_{n-1}$ of type $[2,\ldots , 2]$; the local Chern
  cycle is $0$.

  For cusp singularity, the exceptional divisor consists of a cycle of rational curves
  $C_1,\ldots , C_m$ of type $[b_0,\ldots, b_m]^{\circ}$ for some integer $m\ge 1$; the local Chern
  cycle is $\sum_{i=1}^m C_i$.
\end{rmk}

As we will make frequent comparison to the surfaces associated to the full Hilbert modular group
$\PSL_2(\cO_E)$, we set up some notation to facilitate the analysis. Let
$\Gamma\subset \PGL_2^+(\R)^2$ be commensurable with $\PSL_2(\cO_E)$. Set $Y_\Gamma$ to be the
quotient $\Gamma\lmod\cH^2$ and let $X_\Gamma$ be the minimal desingularisation of
$\overline{\Gamma \lmod \uhp^2}$.  In this notation our Hilbert modular surface $X_0^+ (\fp)$ is
$X_{W\Gamma_0^E (\fp)}$. As is computed on page 64 of \cite{MR930101} we have the following with a
slight change of notation:

\begin{thm}
  The Chern numbers for $X_\Gamma$ are given as follows:
  \begin{align}
    \label{eq:chern_number}
    c_1^2(X_\Gamma)&=2 \vol(\Gamma \lmod \uhp^2)+ c + \sum
                     a(\Gamma;n;a,b)c(n;a,b),\\
    c_2(X_\Gamma)&=\vol(\Gamma\lmod\uhp^2)+l+\sum
                   a(\Gamma;n;a,b)(l(n;a,b)+ \frac{n-1}{n})
  \end{align}
  where $a(\Gamma;n;a,b)$ is the number of quotient singularity of
  $\Gamma\lmod\uhp^2$ of type $(n;a,b)$; for a quotient
  singularity of type $(n;a,b)$, $c(n;a,b)$ is the self-intersection number of the local Chern cycle,
  $l(n;a,b)$ is the number of curves in the resolution; $c$ is the sum of the self-intersection number of the local Chern
  cycles of cusp singularities and $l$ is sum of number of curves in the resolution of cusps.
\end{thm}
We also record a theorem of Siegel on the volume of Hilbert modular varieties. See \cite
[Theorem~IV.1.1] {MR930101}.
\begin{thm}
  Let $E$ be a totally real field of degree $n$ over $\Q$. Let $\omega$ be the invariant volume form
  on $\cH^n$:
  \begin{equation*}
    (-1)^n\frac{1}{(2\pi)^n}\frac{dx_1\wedge dy_1}{y_1^2}\wedge\cdots\wedge\frac{dx_n\wedge dy_n}{y_n^2}.
  \end{equation*}
  Then
  \begin{equation}
    \vol(\PSL_2(\cO_E)\lmod\uhp^2) := \int_{\PSL_2(\cO_E)\lmod\uhp^2} \omega =2\zeta_E(-1).   
  \end{equation}
\end{thm}

Now we will estimate the Chern numbers under the assumption that $D>12$. This ensures that we only
have elliptic points of type $(2;1,1)$ or $(3;1,\pm 1)$ for $\Gamma_0^E(\fp)$ and hence only
elliptic points of type $(2;1,1)$, $(3;1,\pm 1)$, $(4;1,\pm 1)$ or $(6;1,\pm 1)$ for
$W\Gamma_0^E(\fp)$. From \cite[II. 6]{MR930101} as summarised in Remark~\ref{rmk:quot-resol}, we
know how the elliptic points are resolved and can compute the self-intersection number of local
Chern cycles. Thus after we plug in the values, equation \eqref{eq:chern_number} reads
\begin{equation}
  \begin{split}
    c_1^2(X_{W\Gamma_0^E(\fp)})=& \frac{1}{2}[\PSL_2(\cO_E):\Gamma_0^E(\fp)]4\zeta_E(-1)+ c \\
    &-\frac{1}{3}a_3^+ (W\Gamma_0^E(\fp)) -a_4^+(W\Gamma_0^E(\fp))-\frac{8}{3}a_6^+(W\Gamma_0^E(\fp));\\
    c_2(X_{W\Gamma_0^E(\fp)})=& \frac{1}{2}[\PSL_2(\cO_E):\Gamma_0^E(\fp)]2\zeta_E(-1)+l
    +(1+\frac{1}{2})a_2(W\Gamma_0^E(\fp))\\
    &+(1+\frac{2}{3})a_3^+(W\Gamma_0^E(\fp))+(2+\frac{2}{3})a_3^-(W\Gamma_0^E(\fp))
    +(1+\frac{3}{4})a_4^+(W\Gamma_0^E(\fp))\\
    &+(3+\frac{3}{4})a_4^-(W\Gamma_0^E(\fp))+(1+\frac{5}{6})a_6^+(W\Gamma_0^E(\fp))+(5+\frac{5}{6})a_6^-(W\Gamma_0^E(\fp)).
  \end{split}
\end{equation}
  
First we estimate $c_2 (X_{W\Gamma_0^E(\fp)})$.
\begin{lemma}[{\cite[Section VII.5, eq.~(1)]{MR930101}}]
  \label{lemma:zeta-estimate}
  For all fundamental discriminant $D$, $\zeta_E(-1)>\frac{D^{3/2}}{360}$.
\end{lemma}

As $a_2$, $a_3^\pm$, $a_4^\pm$, $a_6^\pm$ and $l$ are non-negative, we get
\begin{prop}
  \label{prop:c_2-estimate}
  $c_2(X_{W\Gamma_0^E(\fp)})>(\norm{\fp}+1)\frac{D^{3/2}}{360}$.
\end{prop}

Now we estimate $c_1^2(X_{W\Gamma_0^E(\fp)})$.  The self-intersection number $c$ of the local Chern
cycle at the cusp is equal to that for $\PSL_2(\cO_E)$ as the isotropy group for the unique cusp in
$W\Gamma_0^E(\fp)$ is the same as that in $\PSL_2(\cO_E)$. Thus we use the results from
\cite{MR930101} directly. There the quantity $c$ is shown to be equal to the negative of the length
of the cycle in \cite [eq.~(7), II.5] {MR930101} and the length satisfies the inequality below
 \cite [eq.~(2), VII.5] {MR930101}:
\begin{equation}
  \label{eq:local_chern_cycle}
  {c> - \frac {1} {2} \sum_{|x|< \sqrt {D}} \sigma_0 \left(\frac{D-x^2}{4}\right):=
    - \frac{1}{2}\sum_{x^2<D, x^2 \equiv D \pmod{4}}\sum_{a>0, a|\frac{D-x^2}{4}}1}.
\end{equation}
Combining the inequality with that in \cite[Lemma VII.5.3]{MR930101}, we get the following
lemma. Note that the condition that $E$ has narrow class number $1$ is in effect.
\begin{lemma}
  \label{lemma:c-estimate}
  The self-intersection number $c$ of the local Chern cycle of the cusp singularity satisfies
  \begin{equation}
    c> -\frac{1}{2} D^{1/2}\left( \frac{3}{2\pi^2}\log^2 (D) + 1.05\log (D) \right).
  \end{equation}
\end{lemma}
Then we estimate $a_2$ and $a_3^\pm$.
\begin{defn}
  Let $h(D)$ denote the class number of the quadratic field $\Q(\sqrt{D})$ where $D$ is a
  fundamental discriminant.
\end{defn}
\begin{lemma} [{\cite[Lemma VII.5.2]{MR930101}}]
  If $-\Delta < -4$ is a fundamental discriminant then
  $h(-\Delta)\le \frac{\sqrt{\Delta}}{\pi}\log \Delta$.
\end{lemma}
The following is from \cite [page 17] {MR930101}. The other cases listed there are ruled out because
$E$ is assumed to have the narrow class number $1$.
\begin{lemma}
  If $D>12$, then
  \begin{equation}
    \begin{split}
      a_2(\PSL_2(\cO_E))= &h(-4D)\\
      a_3^+(\PSL_2(\cO_E))=& \frac{1}{2}h(-3D).
    \end{split}
  \end{equation}
\end{lemma}

Combining the above two lemmas we get:
\begin{lemma}\label{lemma:full-hilbert-mod-gp-ell-pts-estimates}
  If $D>12$, then
  \begin{equation}
    \begin{split}
      a_2(\PSL_2(\cO_E)) &\le \frac{\sqrt{4D}}{\pi}\log (4D)\\
      a_3^+(\PSL_2(\cO_E)) &\le \frac{\sqrt{3D}}{2\pi} \log(3D).
    \end{split}
  \end{equation}
\end{lemma}
Now we put in level structure.
\begin{lemma}
  \label{lemma:Gamma0_ell_pts_estimates}
  If $D>12$ then
  \begin{equation}
    \begin{split}
      a_2(\Gamma_0^E(\fp)) &\le \frac{3\sqrt{4D}}{\pi}\log (4D)\\
      a_3^+(\Gamma_0^E(\fp))&\le \frac{3\sqrt{3D}}{2\pi} \log(3D).
    \end{split}
  \end{equation}
\end{lemma}
\begin{proof}
  Let $z$ be an elliptic point of $\PSL_2(\cO_E)$ with isotropy group generated by
  $g=\bigl( \begin{smallmatrix} a&b\\ c&d
  \end{smallmatrix} \bigr)$. We have coset decomposition of
  $\PSL_2(\cO_E)=\cup_\alpha \Gamma_0^E(\fp)\delta_\alpha \cup\Gamma_0^E(\fp)\delta_\infty $, where
  $\delta_\alpha=\bigl( \begin{smallmatrix} 1&0\\ \alpha&1
  \end{smallmatrix} \bigr)$ with $\alpha\in\cO_E$ running through a set of representatives of
  $\cO_E/\fp$ and $\delta_\infty=\bigl( \begin{smallmatrix} 0&1\\ -1&0
  \end{smallmatrix} \bigr)$. This elliptic point corresponds to several
  $\Gamma_0^E(\fp)$-inequivalent points: $\delta_\alpha z$'s and $\delta_\infty z$. All elliptic
  points for $\Gamma_0^E(\fp)$ must be one of those. To see which ones of $\delta_\alpha z$'s are
  elliptic points for $\Gamma_0^E(\fp)$, we just need to check if
  $\delta_\alpha g \delta_\alpha^{-1}$ is in $\Gamma_0^E(\fp)$, since we are dealing with elliptic
  points of type $(2;1,1)$ and $(3;1,\pm 1)$ only.  This is equivalent to checking if
  $c-(d-a)\alpha-b\alpha^2$ is in $\fp$. In $\F_\fp$, the equation $c-(d-a)\alpha-b\alpha^2=0$ has
  at most two solutions unless $c, a-d, b \in \fp$. Now we claim that it is not possible to have
  $c, a-d, b \in \fp$. If the claim holds then only two of $\delta_\alpha z$'s can be elliptic
  points for $\Gamma_0^E(\fp)$. Adding in the point $\delta_\infty z$, we get at most three elliptic
  points for $\Gamma_0^E(\fp)$ lying over $z$. Thus the number of elliptic points of a given type
  can increase to at most threefold that for $\PSL_2(\cO_E)$. Combining with the inequalities of
  Lemma~\ref{lemma:full-hilbert-mod-gp-ell-pts-estimates}, we get the inequalities in the statement
  of this lemma.

  It remains to prove our claim. Assume otherwise, i.e., $c, a-d, b \in \fp$.  Since $g$ is elliptic
  and $D>12$, we conclude from Sec.~\ref{sec:ell_pt} that $a+d$ can only take the values $0, \pm
  1$. From $ad-bc=1$, we find that $a^2 \equiv 1 \pmod \fp$. Thus $a\equiv \pm 1 \pmod \fp$. We may
  change the matrix that represents $g$ so that $a\equiv 1 \pmod \fp$. Thus also $d\equiv 1\pmod \fp$. Using $ad-bc=1$ again, we must have
  $a+d \equiv 2 \pmod {\fp^2}$. Since the value of $a+d$ is  $0$ or $\pm 1$, $\fp^2$ divides $(2)$ or $(3)$, but $(2)$ and $(3)$ do not ramify in $E$ as $D>12$.
  We get a contradiction.
\end{proof}

\begin{lemma}
  \label{lemma:ell-pt-estimate}
  Suppose $D>12$.  Then
  \begin{enumerate}
  \item $a_4^+(W\Gamma_0^E(\fp)) =0$ unless $(2)$ is inert in $\cO_E$ and $\fp=(2)$;
  \item when $(2)$ is inert in $\cO_E$ and $\fp=(2)$,
    \begin{equation*}
      a_4^+(W\Gamma_0^E(\fp)) \le \frac{3\sqrt{4D}}{\pi}\log (4D);
    \end{equation*}
  \item $a_6^+(W\Gamma_0^E(\fp))=0$ and
    \begin{equation*}
      \frac{1}{3}a_3^+(W\Gamma_0^E(\fp)) \le \frac{\sqrt{3D}}{4\pi} \log(3D)
    \end{equation*}
    unless $(3)$ is inert in $\cO_E$ and $\fp=(3)$;
  \item when $(3)$ is inert in $\cO_E$ and $\fp=(3)$,
    \begin{equation*}
      \frac{1}{3}a_3^+(W\Gamma_0^E(\fp))+\frac{8}{3}a_6^+(W\Gamma_0^E(\fp))
      \le \frac{4\sqrt{3D}}{\pi} \log(3D).
    \end{equation*}
  \end{enumerate}
\end{lemma}
\begin{proof}
  We check the rotation factor \eqref{eq:rotation}
  \begin{equation*}
    \cos \theta_j = \frac{\tr(\alpha_j)}{2\sqrt{\det(\alpha_j)}} 
  \end{equation*}
  associated to an elliptic element $\alpha\in W\Gamma_0^E (\fp)$. In order to have a point with
  isotropy group of order $4$ in $W\Gamma_0^E(\fp)$ we must have
  $\cos \theta_j = \pm\frac{\sqrt{2}}{2}$. As $D>12$, this can only happen when $\fp=(2)$ and
  $\det(\alpha_j)= \varpi_\fp$ modulo squares in $\cO_E$.  In order to have a point with isotropy
  group of order $6$ in $W\Gamma_0^E(\fp)$ we must have $\cos \theta_j = \pm\frac{\sqrt{3}}{2}$. As
  $D>12$, this can only happen when $\fp=(3)$ and $\det(\alpha_j)= \varpi_\fp$ modulo squares in
  $\cO_E$.
  
  The Atkin-Lehner operator $w$ exchanges some of the $\Gamma_0^E(\fp)$-inequivalent
  $(3;1,1)$-points which result in $(3;1,1)$-points for $W\Gamma_0^E(\fp)$ and fixes the rest of the
  points which result in $(6;1,1)$-points for $W\Gamma_0^E(\fp)$. Thus we get
  \begin{equation}
    2a_3^+(W\Gamma_0^E(\fp))+a_6^+(W\Gamma_0^E(\fp))= a_3^+(\Gamma_0^E(\fp)).
  \end{equation}
  It is easy to see that
  \begin{equation*}
    \frac{1}{3}a_3^+(W\Gamma_0^E(\fp))+\frac{8}{3}a_6^+(W\Gamma_0^E(\fp)) \le \frac{8}{3}a_3^+(\Gamma_0^E(\fp)). 
  \end{equation*}
  Combining with Lemma \ref{lemma:Gamma0_ell_pts_estimates}, we get our estimate.

  The Atkin-Lehner operator exchanges some of the $\Gamma_0^E(\fp)$-inequivalent $(2;1,1)$-points
  which result in $(2;1,1)$-points for $W\Gamma_0^E(\fp)$ and fixes the rest of the points which
  result in $(4;1,1)$-points for $W\Gamma_0^E(\fp)$. All $(4;1,1)$-points for $W\Gamma_0^E(\fp)$
  arise in this way, but we may get extra $(2;1,1)$-points $W\Gamma_0^E(\fp)$ not arising in this
  way. Thus we have
  \begin{equation}
    a_4^+(W\Gamma_0^E(\fp))\le a_2^+(\Gamma_0^E(\fp)).
  \end{equation}
  Combining with Lemma \ref{lemma:Gamma0_ell_pts_estimates}, we get our estimate.
\end{proof}

Combining all these inequalities (Lemmas~\ref{lemma:zeta-estimate}, \ref{lemma:c-estimate},
\ref{lemma:ell-pt-estimate}) we finally arrive at an estimate for $c_1^2(X_0^+(\fp))$.
\begin{prop}\label{prop:c_1-estimate}
  Suppose $D>12$. Then
  \begin{equation}
    \begin{split}
      \label{eq:c1^2-estimate}
      c_1^2(X_0^+(\fp)) &> (\norm\fp +1)\frac{D^{3/2}}{180} -
      \frac{1}{2}D^{1/2}(\frac{3}{2\pi^2}\log^2 D
      +1.05\log D) \\
      &-
      \begin{cases}
        \frac{1}{4\pi}\sqrt{3D}\log(3D) &\text{if $\fp \ne
          (3)$}\\
        \frac{4}{\pi}\sqrt{3D}\log(3D) &\text{if $\fp = (3)$}
      \end{cases}\\
      &-
      \begin{cases}
        0 &\text{if $\fp \ne     (2)$}\\
        \frac{3}{\pi}\sqrt{4D} \log(4D) &\text{if $\fp = (2)$}.
      \end{cases}
    \end{split}
  \end{equation}
\end{prop}
Now that we have inequalities for $c_1^2$ and $c_2$ of the Hilbert modular surfaces, we can check
for what values of $D$ and $n$ these are of general type. For a given $D$, we may bound $c$ more precisely
by using \eqref{eq:local_chern_cycle}.

\begin{thm}\label{thm:general_type}
  Suppose $D>12$ and $\Cl^+(\Q(\sqrt{D}))=1$. Set $n=\norm \fp+1$. Then the Hilbert modular surface
  $X_0^+(\fp)$ is of general type if $D$ or $n$ is sufficiently large or more precisely if the
  following conditions on $D$ and $n$
  are satisfied:\\
  \begin{center}
    \begin{longtable}{|p{7cm}|c|}
      \hline $D\ge 853$ or $D=313$, $337$, $353$, $409$, $433$, $449$, $457$, $521$, $569$, $593$,
      $601$, $617$, $641$, $653$, $661$, $673$, $677$, $701$, $709$, $757$, $769$, $773$, $797$,
      $809$, $821$, $829$ & no constraint on $n$\\\hline $D=241$&$n>3$\\\hline $D=193$&$n>3$\\\hline
      $D=157$, $ 181$, $ 277$, $ 349$, $ 373$, $ 397$, $ 421$, $ 541$, $ 613$ &
      $\fp\neq (2)$\\\hline $D=233$, $281$ & $\fp\neq (3)$\\\hline $D=149$, $173$, $197$, $269$,
      $293$, $317$, $389$, $461$, $509$, $557$ & $\fp\neq (2), (3)$\\\hline $D=137$ & $n>3$ and
      $\fp\neq (3)$\\\hline $D=113$ &$n>4$ and $\fp\neq (3)$\\\hline $D=109$&$n>4$ and
      $\fp\neq (2)$\\\hline $D=101$&$n>3$ and $\fp\neq (2), (3) $\\\hline $D=97$&$n> 6$\\\hline
      $D=89$& $n>5$ and $\fp\neq (3)$\\\hline $D=73$& $n>7$\\\hline $D=61$& $n>6$\\\hline $D=53$&
      $n>7$ and $\fp\neq (3)$\\\hline $D=41$& $n>12$ \\\hline $D=37$& $n>12$ \\\hline $D=29$&
      $n>15$\\\hline $D=17$& $n>32$ \\\hline
      $D=13$&$n>41$  \\
      \hline
    \end{longtable}
  \end{center}
\end{thm}
\begin{proof}
  We note that $n=\norm\fp +1 \ge 3$. By Prop.~\ref{prop:c_2-estimate}, as long as $D>127$,
  $c_2 (X_{W\Gamma_0^E(\fp)})>12$. Next we give a rough estimate for $D$ so that
  $c_1^2 (X_{W\Gamma_0^E(\fp)}) >0$ by using the inequality \eqref{eq:c1^2-estimate}.  When
  $\fp \neq (2)$ or $(3)$, as long as $D>414$, we have $c_1^2 (X_{W\Gamma_0^E(\fp)}) >0$.  When
  $\fp = (2)$, as long as $D>849$, we have $c_1^2 (X_{W\Gamma_0^E(\fp)}) >0$. When $\fp = (3)$, as
  long as $D>384$, we have $c_1^2 (X_{W\Gamma_0^E(\fp)}) >0$. The numerical computation was done in
  SageMath\cite{sagemath}.  We also used it to produce a list of discriminants of real quadratic
  fields with narrow class number $1$. Here is the list up to $853$ which is the smallest one that is greater
  than $849$:
  \begin{align*}
    &5, 8, 13, 17, 29, 37, 41, 53, 61, 73, 89, 97, 101, 109, 113, 137, 149, 157, 173, 181,\\
    &193, 197, 233, 241, 269, 277, 281, 293, 313, 317, 337, 349, 353, 373, 389, 397, 409,\\
    &421, 433, 449, 457, 461, 509, 521, 541, 557, 569, 593, 601, 613, 617, 641, 653, 661,\\
    &673, 677, 701, 709, 757, 769, 773, 797, 809, 821, 829, 853.      
  \end{align*}
  Thus for $D\ge 853$ and any $\fp$ we always have $c_2 (X_{W\Gamma_0^E(\fp)})>12$ and
  $c_1^2 (X_{W\Gamma_0^E(\fp)}) >0$. By Prop.~\ref{prop:criterion_general_type}, these are surfaces
  of general type.

  Next we compute for each of the discriminant $D$ in the list, a sufficient condition on $n$ (or on
  $\fp$) so that $c_1^2 (X_{W\Gamma_0^E(\fp)})+c_2 (X_{W\Gamma_0^E(\fp)})>12$ and
  $c_1^2 (X_{W\Gamma_0^E(\fp)}) >0$ are satisfied. We may use the sharper bound
  \eqref{eq:local_chern_cycle} for the intersection number $c$ of local Chern cycles. We note that
  the formula \eqref{eq:c1^2-estimate} for estimating $c_1^2$ branches when we have $\fp= (2), (3)$.
  Under the constraint of our theorem, $(2)$ is split if and only if $D\equiv 1 \pmod 8$, $(2)$ is
  inert if and only if $D\equiv 5 \pmod 8$; $(3)$ is split if and only if $D\equiv 1 \pmod 3$ and
  $(3)$ is inert if and only if $D\equiv 2 \pmod 3$. For those $D$'s with inert primes $(2)$ or
  $(3)$, we compute the values of $c_1^2$ and $c_2$ to check if we get surfaces of general type or
  not. The numerical results are summarised in the table.
\end{proof}
\begin{rmk}
  We are providing a sufficient condition for the Hilbert modular surface
  $X_0^+ (\fp)= X_{W\Gamma_0^E(\fp)}$ to be of general type. More precise analysis is needed to
  determine the exact type of a given $X_0^+ (\fp)$. We give some examples in the
  Sec.~\ref{sec:examples}.
\end{rmk}

\subsection{Examples}
\label{sec:examples}
The  two examples give rational surfaces.
\subsubsection{$D=5$}
Let $E=\Q(\sqrt{5})$. The Galois conjugation is denoted by ${}^*$.  By the algorithm\cite [eq.~(3)
on page 38] {MR930101}, the cusp resolution at infinity of $\bar {Y}_0^+ (\fp)$ is a cycle of type
$[3]^\circ$, i.e., a rational curve with an ordinary double point and with self-intersection number
$-1$. To find the $W\Gamma_0^E (\fp)$-inequivalent elliptic points, we first consider the
$\PSL_2(\cO_E)$-inequivalent elliptic points which were worked out in \cite [Satz~1]
{MR0229579}. Let $\varepsilon= (1+\sqrt {5})/2$ and $\varepsilon^*= (1-\sqrt {5})/2$.  We list the
type of the elliptic point and a generator of the isotropy group:
\begin{align*}
  (2;1,1), & \smatrix{0}{1}{-1}{0} \quad
  & (2;1,1), & \smatrix{0}{-\varepsilon^*}{-\varepsilon}{0}\\
  (3;1,1), & \smatrix{0}{1}{-1}{1} \quad
  & (3;1,-1), & \smatrix{0}{-\varepsilon^*}{-\varepsilon}{1}\\
  (5;1,3), & \smatrix{0}{1}{-1}{\varepsilon} \quad
  & (5;1,2), & \smatrix{1}{-\varepsilon^*}{\varepsilon^*}{-\varepsilon^*}.
\end{align*}

We have coset decomposition
$\PSL_2(\cO_E) = \cup_{\alpha}\Gamma_0^E (\fp)\delta_\alpha \cup \Gamma_0^E (\fp)\delta_\infty$
where $\delta_\alpha=\bigl( \begin{smallmatrix} 1&0\\ \alpha&1
\end{smallmatrix} \bigr)$ with $\alpha\in\cO_E$ running through a set of representatives of
$\cO_E/\fp$ and $\delta_\infty=\bigl( \begin{smallmatrix} 0&1\\ -1&0
\end{smallmatrix} \bigr)$. Let $z$ be an elliptic point for $\PSL_2(\cO_E)$ and
$\gamma\in\PSL_2(\cO_E)$ be a generator of the isotropy group. Then $\delta_\alpha.z$
(resp. $\delta_\infty.z$) is an elliptic points for $\Gamma_0^E(\fp)$ if and only if
$\delta_\alpha\gamma\delta_\alpha^{-1}$ (resp. $\delta_\infty\gamma\delta_\infty^{-1}$) lies in
$\Gamma_0^E(\fp)$. Write $\gamma$ as $\smatrix{a}{b}{c}{d}$. It is easy to check that
$\delta_\alpha\gamma\delta_\alpha^{-1}\in\Gamma_0^E(\fp)$ if and only if
$b\alpha^2+ (d-a)\alpha -c\in\fp$ and that $\delta_\infty\gamma\delta_\infty^{-1}\in\Gamma_0^E(\fp)$
if and only if $b\in\fp$. Once we get the $\Gamma_0^E (\fp)$-inequivalent elliptic points, we need
to check how the Atkin-Lehner operator $w$ acts on them. We will work this out with a more specific
$\fp$.

\paragraph{$\fp=(2)$}
The Hilbert modular surface $X_0^+ (\fp)$ is a rational surface. We explain below.  We get the
following inequivalent elliptic points for $\Gamma_0^E(\fp)$. Instead of writing out their coordinates,
we write down the type and a generator of the isotropy subgroup of $\Gamma_0^E(\fp)$ that fixes each
elliptic point. The coordinates can be recovered by \eqref{eq:ell-pt-coords}.
\begin{align*}
  (2;1,1), & \bmat{-1}{1}{-2}{1} \quad & (2;1,1), & \bmat{-1}{-\varepsilon^*}{-2\varepsilon}{1}\\
  (3;1,1), & \bmat{-\varepsilon}{1}{-2 (1+\varepsilon)}{1+\varepsilon} \quad & (3;1,1), & \bmat{-\varepsilon^*}{1}{-2 (1+\varepsilon^*)}{1+\varepsilon^*}\\
  (3;1,-1), & \bmat{\varepsilon^*}{-\varepsilon^*}{-2\varepsilon}{\varepsilon} \quad & (3;1,-1), &\bmat{1+\varepsilon^*}{-\varepsilon^*}{2\varepsilon^*}{-\varepsilon^*}.
\end{align*}
It can be checked directly that the Atkin-Lehner operator $w$ fixes the two $(2;1,1)$-points
respectively. Since there cannot exist elliptic points of type $(6;1,\pm1)$ for $W\Gamma_0^E (\fp)$,
we see that $w$ must exchange the two $(3;1,1)$- (resp. $(3;1,-1)$-) points. We get one $(4;1,1)$-,
one $(4;1,-1)$-, one $(3;1,1)$-, one $(3;1,-1)$- and possibly some new $(2;1,1)$-points.

We consider certain Hirzebruch cycles on the Hilbert modular surface. Set
\begin{equation}\label{eq:F_B}
  \tilde{F}_B=\left\lbrace (z_1,z_2) \in \cH^2  \cup \PP^1(E) :
    \begin{pmatrix}
      z_2&1
    \end{pmatrix} B
    \begin{pmatrix}
      z_1\\1
    \end{pmatrix}=0 \right\rbrace
\end{equation}
where $B$ is a skew-Hermitian matrix in $M_2 (E)$, i.e., $\trpz {B^*} = B$.  Let $F_B$ denote the
strict transform in $X_0^+(\fp)$ of the image of $\tilde {F}_B$ in $Y_0^+ (\fp)$.

Let $B=\smatrix{0}{\sqrt {5}\varepsilon^*}{\sqrt {5}\varepsilon}{0}$. The $(3;1,-1)$-point can be
represented by $(-(\sqrt {5}+ i\sqrt {3})\varepsilon^*/4, (\sqrt {5}+ i\sqrt {3})\varepsilon/4)$, so
obviously it lies on $F_B$. The $(4;1,-1)$-point can be represented by
$((1+i)/ (2\varepsilon), (-1+i)/ (-2\varepsilon^*))$.  After applying translation by
$\smatrix{1}{-1}{0}{1}$, we get the $\Gamma_0^E(\fp)$-equivalent elliptic point
$((-\sqrt {5}+i)/ (2\varepsilon), (-\sqrt {5}+i)/ (-2\varepsilon^*))$. Thus we see that the
$(4;1,-1)$-point also lies on $F_B$.  The stabiliser $\Gamma_B$ of $\tilde {F}_B$ in
$\Gamma_0^E(\fp)$ consists of those elements $\gamma$ such that $\trpz{\gamma^*}B\gamma = \pm
B$. Thus $\Gamma_B$ is the degree $2$ extension of the group
\begin{equation}\label{eq:Gamma_B-base}
  \left\{
    \smatrix{a}{b}{c}{d}\in \Gamma_0^E(\fp):
    a, d \in \Z, c\in
    2\varepsilon\sqrt{5}\Z, b\in \varepsilon^*\sqrt{5}\Z
  \right\}
\end{equation}
generated by $\smatrix{\sqrt {5}}{2\varepsilon^*}{2\varepsilon}{\sqrt {5}}$.  The stabiliser
$\tilde {\Gamma}_B$ of $\tilde {F}_B$ in $W\Gamma_0^E(\fp)$ is a degree $2$ extension of $\Gamma_B$
by $\smatrix{2}{\varepsilon^*\sqrt {5}}{2\varepsilon\sqrt {5}}{-4}$. Note that the group
\eqref{eq:Gamma_B-base} is isomorphic to $\Gamma_0^{\Q} (10)$ which is the congruence subgroup of
$\SL_2 (\Z)$ with lower-left entry congruent to $0$ modulo $10$. As
$\Gamma_0^{\Q} (10\Z)\lmod\uhp \cup \PP^1 (\Q)$ is isomorphic to $\PP^1$, the non-singular model of
$F_B$ is isomorphic to $\PP^1$. Let $sw$ denote the involution on $\bar {Y}_0^+ (\fp)$ induced by
swapping coordinates on $\cH^2$: $(z_1,z_2)\mapsto (z_2,z_1)$. It can be extended to an involution
on $X_0^+ (\fp)$. If $(z_1,z_2)$ is a point satisfying \eqref{eq:F_B} then $(z_2,z_1)$ is
$\Gamma_0^E (\fp)$-equivalent to $(z_1,z_2)$ via $\smatrix{-\varepsilon^*}{}{}{\varepsilon}$. This
means that $sw$ fixes $F_B$ point-wise. This, in turn, implies that $F_B$ is non-singular. We
conclude that $F_B$ is a non-singular rational curve.

Recall that the cusp resolution is formed by gluing copies of $\CC^2$. Following the method in
\cite[V.2]{MR930101} we can determine the local equation of $F_B$ on each copy of $\CC^2$.  On the
$k$-th copy of $\CC^2$ the coordinates are related by
\begin{equation*}
  2\pi i z_j = A_{k-1}^{(j)} \log (u_k) + A_k^{(j)} \log (v_k)
\end{equation*}
where $z_j$ denote the coordinate on the $j$-th copy of $\cH$ for $j=1,2$ and $(u_k,v_k)$ denote the
coordinates of the $k$-th copy of $\CC^2$ for $k\in \Z$. For the case at hand, we take $A_0=1$ and
$A_1= (3-\sqrt {5})/2$ which form a $Z$-basis of $\cO_E$. Other values of $A_k$'s  are omitted. Recall that $a^{(j)}$ denote the image of $a\in E$ via the $j$-th embedding to $\RR$. Then on the $1$-st copy of $\CC^2$, the
equation of $F_B$ becomes $u_1=v_1$. There is no intersection with coordinate axes in other copies
of $\CC^2$. Thus $F_B$ intersects the cusp resolution at the origin of the $1$-st copy of
$\CC^2$ which corresponds to the ordinary double point on the cusp resolution. Thus the intersection
number of $F_B$ with the cusp resolution is $2$.  By \cite[Corollary~4.1]{MR930101}, we get
\begin{equation*}
  c_1(X_0^+(\fp)).F_B= 2 \vol (\tilde {\Gamma}_B\lmod \cH) + \sum_x Z_x. F_B
\end{equation*}
where the sum runs over all singularities $x$ of $Y_0^+ (\fp)$ and $Z_x$ denotes the local Chern
cycle of $x$.  The volume
\begin{equation*}
  \vol (\tilde {\Gamma}_B\lmod \cH) = \frac{1}{4}\vol (\Gamma_0^\Q (10)\lmod \cH) = \frac{18}{4}\vol (\SL_2 (\Z)\lmod \cH) = 9\zeta_\Q (-1) =-\frac{3}{4}.
\end{equation*}
Here $\zeta_\Q$ denotes the Riemann zeta function.  The local Chern cycles needed in the computation
can be looked up in Remark~\ref{rmk:local-chern-cycle}.  Thus we find
\begin{equation*}
  c_1(X_0^+(\fp)).F_B= - \frac{3}{2} 
  +\frac{1}{3}\cdot n_3 +\frac{1}{2}\cdot n_4
\end{equation*}
where $n_3$ is the number of $(3;1,1)$-points that $F_B$ passes through and $n_4$ is the number of
$(4;1,1)$-points that $F_B$ passes through.  As intersection numbers are integers, we are force to
have $n_3=0$ and $n_4=1$ and thus $c_1(X_0^+(\fp)).F_B=1$. By the Adjunction formula, $F_B^2=-1$.
We get a linear configuration of non-singular rational curves with self-intersection numbers $-2$,
$-1$, $-2$, where the $(-2)$-curves come from desingularity of the $(3;1,-1)$- and the
$(4;1,-1)$-points mentioned above. After blowing down $F_B$ we acquire two intersecting
$(-1)$-curves and this shows that the surface $X_0^+(\fp)$ is a rational surface by the rationality
criterion\cite [VII.2.2] {MR930101}.

\subsubsection{$D=13$}
We adopt essentially the same notation as in the previous example. Now the quadratic field is
$E=\Q(\sqrt{13})$. The Galois conjugation is denoted by ${}^*$. We will regard $E$ as a subfield of
$\R$. Set $\varepsilon= (3+\sqrt {13})/2$ to be a fundamental unit. The cusp resolution at the
infinity of $\bar {Y}_0^+ (\fp)$ is of a configuration of type $[5,2,2]^\circ$. We label the
non-singular rational curves occurring in the cusp resolution as $S_0$, $S_1$ and $S_2$. Following
the method in \cite{MR0229579}, we can locate all the $\PSL_2(\cO_E)$-inequivalent elliptic
points. We review the process briefly. First we can compute that the $y$-coordinates of an elliptic
fixed point $z= (z_1,z_2)$ in the fundamental domain given as in \cite{MR0229579} satisfies
\begin{equation*}
  1\le \left (\frac{17}{8}\right)^2 +  (y_1y_2)^2 + \frac{13}{16}(y_1y_2),
\end{equation*}
or, in other words,
\begin{equation}\label{eq:bound-Ny}
  y_1y_2 \ge (-17 + 2\sqrt {94})/16 > 0.149.
\end{equation}

Consider the elliptic fixed points of order $2$. Assume it is fixed by the matrix
$\smatrix{a}{b}{c}{d}\in \PSL_2 (\cO_E)$. We may assume that $c<0$. From \eqref{eq:bound-Ny} and
\eqref{eq:ell-pt-coords}, we deduce that $|c c^*|\le 6$. Thus up to a unit $c$ is either $1$,
$4-\sqrt {13}$ or $2$. In the fundamental domain, we have
$\varepsilon^{-2} \le y_1/y_2 < \varepsilon^2$ and $(x_1,x_2)$ lies in the set
\begin{equation*}
  \fP:=\{(u + v\sqrt {13}, u - v\sqrt {13}) | -1/2 < u \le 1/2, -1/4 < v \le 1/4\}.
\end{equation*}
Thus $c$ can take the following values:
\begin{equation}\label{eq:c-vals}
  -1, -\varepsilon, (1-\sqrt {13})/2, (-5-\sqrt {13})/2, -2, -2\varepsilon.
\end{equation}
For each of these values we find all values $a$ and $d$ in $\cO_E$ such that $a+d=0$ and
$((a-d)/2c, (a^*-d^*)/2c^*)$ lies in the set $\fP$. Then we determine the value for $b\in \cO_E$ by
ensuring the determinant is $1$.

Finally we need to check which ones are $\PSL_2 (\cO_E)$-conjugate matrices and keep only one of
those. We summarise the results below. A set of $\PSL_2 (\cO_E)$-inequivalent $(2;1,1)$-points is given by
the fixed point of the following matrices:
\begin{equation*}
  \bmat{0}{1}{-1}{0}, \quad \bmat{0}{-\varepsilon^*}{-\varepsilon}{0}.
\end{equation*}

Next we consider the elliptic fixed points of order $3$. We use the same notation as in the case of
order $2$. Again $c$ can only take the values in \eqref{eq:c-vals}. A similar process produces the
matrices whose fixed points form a set of $\PSL_2 (\cO_E)$-inequivalent points of order $3$. We can
check which ones are $(3;1,1)$-points and which ones are $(3;1,-1)$-points by \eqref{eq:rotation}. A
set of $\PSL_2 (\cO_E)$-inequivalent $(3;1,1)$-points is given by the fixed point of the following matrices:
\begin{equation*}
  \bmat {0} {1} {-1} {1}, \quad \bmat {\varepsilon} {2} {-1-\varepsilon} {1-\varepsilon};
\end{equation*}
a set of $\PSL_2 (\cO_E)$-inequivalent $(3;1,-1)$-points is given by the fixed point of the following
matrices:
\begin{equation*}
  \bmat {-\varepsilon} {2 (\varepsilon-1)} {-\varepsilon} {\varepsilon+1},\quad \bmat {-1} {\varepsilon-1} {\varepsilon^*-1} {2}.
\end{equation*}

It is easy to find the $\Gamma_0^E(\fp)$-inequivalent elliptic points from right coset decomposition
$\PSL_2(\cO_E) = \cup_{\alpha}\Gamma_0^E (\fp)\delta_\alpha \cup \Gamma_0^E (\fp)\delta_\infty$
where $\delta_\alpha=\bigl( \begin{smallmatrix} 1&0\\ \alpha&1
\end{smallmatrix} \bigr)$ with $\alpha\in\cO_E$ running through a set of representatives of
$\cO_E/\fp$ and $\delta_\infty=\bigl( \begin{smallmatrix} 0&1\\ -1&0
\end{smallmatrix} \bigr)$.

\paragraph{$\fp=(4+\sqrt{13})$}
We list the type and one generator of isotropy group for each $\Gamma_0^E(\fp)$-inequivalent
elliptic point:
\begin{align*}
  (3;1,1),&\bmat {-1} {1} {-3} {2} \quad  &(3;1,1) , &\bmat {1-\varepsilon^*} {2} {-1-\varepsilon^*} {\varepsilon^*}\\
  (3;1,-1),&\bmat {1+\varepsilon} {\varepsilon} {-2 (\varepsilon-1)} {-\varepsilon} \quad  &(3;1,-1),&\bmat {2} {1-\varepsilon^*} {1-\varepsilon} {-1}.
\end{align*}
There is no $(2;1,1)$-point.  Since there cannot be any elliptic points with isotropy group of order
$6$ for $W\Gamma_0^E(\fp)$ acting on $\cH^2$, the Atkin-Lehner operator $w$ must exchange the two
$(3;1,1)$-points (resp. $(3;1,-1)$-points).

Now consider the curve $\tilde {F}_B$ on $W\Gamma_0^E(\fp)\lmod \cH^2$ defined as in \eqref{eq:F_B}
and set $B=\bigl(
\begin{smallmatrix}
  0& 4-\sqrt{13}\\ -4-\sqrt{13}&0
\end{smallmatrix}\bigr)$. Define $F_B$ analogously.
The stabiliser $\tilde {\Gamma}_B$ of $\tilde{F}_B$ in $W\Gamma_0^E(\fp)$ consists of elements of
the form
\begin{equation*}
  \left\{\smat {a} {b} {c} {d}\in \Gamma_0^E(\fp): a, d \in \Z, c\in(4+\sqrt{13})\Z, b\in (4-\sqrt{13})\Z \right\}.
\end{equation*}
Thus we find that $F_B$ is birational to $\Gamma_0^\Q(3\Z)\lmod\cH$ which is of genus $0$. The
Atkin-Lehner operator sends $\tilde {F}_B$ to $\tilde {F}_{B'}$ with $B'=\smat {0} {1} {-1}
{0}$. The latter is obviously point-wise stable under the swapping operator $sw$ which is the
involution on $\bar {Y}_0^+ (\fp)$ induced by swapping coordinates on $\cH^2$:
$(z_1,z_2)\mapsto (z_2,z_1)$. Thus $F_B$ is a non-singular rational curve. We can compute how $F_B$
intersects with the cusp resolution. Following the notation of \cite[V.2]{MR930101}, we have
$A_{-1}= (5+\sqrt {13})/2$, $A_0=1$, $A_1= (5-\sqrt {13})/2$ and $A_2=4-\sqrt {13}$. Then $F_B$ has
local equation $u_2 =1$ on the $2$-nd copy of $\CC^2$ and $F_{B'}$ has local equation $u_0=1$ in the
$0$-th copy of $\CC^2$. Thus the intersection number of $F_B$ with the cusp resolution is $2$.

As before we have
\begin{equation}
  c_1(X_0^+(\fp)).F_B= 2\vol(F_B') + \sum Z_x.F_B
\end{equation}
where $Z_x$ is the local Chern cycle at a singular point $x$. Thus we get
\begin{align*}
  c_1(X_0^+(\fp)).F_B &= 2 \vol (\Gamma_0^\Q(3\Z)\lmod\cH) + 2 + \frac{1}{3}\cdot  n_3 \\
                      &= - \frac{4}{3} +2 + \frac{1}{3}\cdot  n_3
\end{align*}
with $n_3$ the number of $(3;1,1)$-points that $F_B$ passes through. As there is just one
$(3;1,1)$-point, we are forced to have $n_3=1$ and thus $c_1(X_0^+(\fp)).F_B=1$. By Adjunction
formula $F_B^2=-1$. Thus we get a linear configuration of $[-2,-2,-1,-3]$ where the $(-2)$-curves
are $S_1$ and $S_2$ from the cusp resolution, the $(-1)$-curve is $F_B$ and the $(-3)$-curve is from the
resolution of singularity of the $(3;1,1)$-point. After blowing down $F_B$ and $S_2$ consecutively,
we get two intersecting $(-1)$-curves. Again by the rationality criterion\cite [VII.2.2] {MR930101},
we conclude that $W\Gamma_0^E((4+\sqrt{13}))\lmod \cH^2$ is a rational surface.

\section*{Acknowledgement}
\label{sec:acknowledgement}

I would like thank my thesis adviser, Professor Shou-Wu Zhang, for all the encouragement and
discouragement during the preparation of this manuscript.

\bibliographystyle{mrl}

\begin{thebibliography}{10}

\bibitem{MR2030225}
W.~P. Barth, K.~Hulek, C.~A.~M. Peters, and A.~Van~de Ven, Compact complex
  surfaces, Vol.~4 of \emph{Ergebnisse der Mathematik und ihrer Grenzgebiete.
  3. Folge. A Series of Modern Surveys in Mathematics [Results in Mathematics
  and Related Areas. 3rd Series. A Series of Modern Surveys in Mathematics]},
  Springer-Verlag, Berlin, second edition (2004), ISBN 3-540-00832-2.

\bibitem{MR3137398}
C.-L. Chai, B.~Conrad, and F.~Oort, Complex multiplication and lifting
  problems, Vol. 195 of \emph{Mathematical Surveys and Monographs}, American
  Mathematical Society, Providence, RI (2014), ISBN 978-1-4704-1014-8.

\bibitem{sagemath}
T.~S. Developers, {S}ageMath, the {S}age {M}athematics {S}oftware {S}ystem
  ({V}ersion 7.5) (2017). {\tt http://www.sagemath.org}.

\bibitem{MR2058644}
N.~D. Elkies, \emph{On elliptic {$K$}-curves}, in Modular curves and abelian
  varieties, Vol. 224 of \emph{Progr. Math.}, 81--91, Birkh\"auser, Basel
  (2004).

\bibitem{MR1643280}
J.~Gonz{\'a}lez and J.-C. Lario, \emph{Rational and elliptic parametrizations
  of {$\bold Q$}-curves}, J. Number Theory \textbf{72} (1998), no.~1,  13--31.

\bibitem{MR563921}
B.~H. Gross, Arithmetic on elliptic curves with complex multiplication, Vol.
  776 of \emph{Lecture Notes in Mathematics}, Springer, Berlin (1980), ISBN
  3-540-09743-0. With an appendix by B. Mazur.

\bibitem{MR2916969}
X.~Guitart, \emph{Abelian varieties with many endomorphisms and their
  absolutely simple factors}, Rev. Mat. Iberoam. \textbf{28} (2012), no.~2,
  591--601.

\bibitem{MR2588775}
X.~Guitart and S.~Molina, \emph{Parametrization of abelian {$K$}-surfaces with
  quaternionic multiplication}, C. R. Math. Acad. Sci. Paris \textbf{347}
  (2009), no. 23-24,  1325--1330.

\bibitem{MR0229579}
K.-B. Gundlach, \emph{Die {F}ixpunkte einiger {H}ilbertscher {M}odulgruppen},
  Math. Ann. \textbf{157} (1965) 369--390.

\bibitem{MR0364262}
F.~Hirzebruch and A.~Van~de Ven, \emph{Hilbert modular surfaces and the
  classification of algebraic surfaces}, Invent. Math. \textbf{23} (1974)
  1--29.

\bibitem{MR0480356}
F.~Hirzebruch and D.~Zagier, \emph{Classification of {H}ilbert modular
  surfaces}  (1977) 43--77.

\bibitem{MR0393045}
F.~E.~P. Hirzebruch, \emph{Hilbert modular surfaces}, Enseignement Math. (2)
  \textbf{19} (1973) 183--281.

\bibitem{MR2480604}
C.~Khare and J.-P. Wintenberger, \emph{On {S}erre's conjecture for
  2-dimensional mod {$p$} representations of {${\rm Gal}(\overline{\Bbb Q}/\Bbb
  Q)$}}, Ann. of Math. (2) \textbf{169} (2009), no.~1,  229--253.

\bibitem{MR2514037}
D.~Mumford, Abelian varieties, Vol.~5 of \emph{Tata Institute of Fundamental
  Research Studies in Mathematics}, Published for the Tata Institute of
  Fundamental Research, Bombay (2008), ISBN 978-81-85931-86-9; 81-85931-86-0.
  With appendices by C. P. Ramanujam and Yuri Manin, Corrected reprint of the
  second (1974) edition.

\bibitem{MR2058652}
E.~E. Pyle, \emph{Abelian varieties over {$\Bbb Q$} with large endomorphism
  algebras and their simple components over {$\overline{\Bbb Q}$}}, in Modular
  curves and abelian varieties, Vol. 224 of \emph{Progr. Math.}, 189--239,
  Birkh\"auser, Basel (2004).

\bibitem{MR1770611}
J.~Quer, \emph{{${\bf Q}$}-curves and abelian varieties of {${\rm
  GL}_2$}-type}, Proc. London Math. Soc. (3) \textbf{81} (2000), no.~2,
  285--317.

\bibitem{MR0393100}
I.~Reiner, Maximal orders, Academic Press [A subsidiary of Harcourt Brace
  Jovanovich, Publishers], London-New York (1975). London Mathematical Society
  Monographs, No. 5.

\bibitem{MR1212980}
K.~A. Ribet, \emph{Abelian varieties over {${\bf Q}$} and modular forms}, in
  Algebra and topology 1992 ({T}aej\u on), 53--79, Korea Adv. Inst. Sci. Tech.,
  Taej\u on (1992).

\bibitem{MR0156001}
G.~Shimura, \emph{On analytic families of polarized abelian varieties and
  automorphic functions}, Ann. of Math. (2) \textbf{78} (1963) 149--192.

\bibitem{MR930101}
G.~van~der Geer, Hilbert modular surfaces, Vol.~16 of \emph{Ergebnisse der
  Mathematik und ihrer Grenzgebiete (3) [Results in Mathematics and Related
  Areas (3)]}, Springer-Verlag, Berlin (1988), ISBN 3-540-17601-2.

\end{thebibliography}

\end{document}